\documentclass[10pt,a4paper]{article}

\usepackage{amsthm}
\usepackage{amsmath}
\usepackage{amssymb}
\usepackage{empheq}
\usepackage{stackrel}
\usepackage{color}
\usepackage{framed}
\usepackage{multirow}
\usepackage{float}

\newcommand{\ds}{\displaystyle}

\newcommand{\lt}{\left}
\newcommand{\rt}{\right}

\def\Vol{\operatorname{Vol}}
\def \CC{\mathbb C}
\def \Re{\operatorname{Re}}

\newtheorem{conjecture}{Conjecture}
\newtheorem{theorem}{Theorem}
\newtheorem{lemma}{Lemma}

\newtheorem{corollary}{Corollary}
\newtheorem{remark}{Remark}

\def\SS{{\mathbb S}}
\def\CC{{\mathbb C}}
\def\RR{{\mathbb R}}
\def\ZZ{{\mathbb Z}}
\def\NN{{\mathbb N}}

\def\Vol{\operatorname{Vol}}
\def\CS{\operatorname{CS}}
\def \Re{\operatorname{Re}}
\def \im{\operatorname{Im}}

\def \Li{\operatorname{Li_2}}
\def \im{\operatorname{Im}}

\def \Arg{\operatorname{Arg}}
\def \det{\operatorname{det}}
\def \Hess{\operatorname{Hess}}

\title{Asymptotics of some quantum invariants\\ of the Whitehead chains}

\author{Ka Ho WONG}

\date{}

\begin{document}

\maketitle

\begin{abstract}
In this paper, we study the asymptotics of the colored Jones polynomials of the Whitehead chains with one belt colored by $M_1$ and all the clasps colored by $M_2$ evaluated at the $(N+1/2)$-th root of unity $t=e^{\frac{2\pi i}{N+1/2}}$, where $M_1$ and $M_2$ are sequences of integers in $N$. By considering the limiting ratios, $s_1$ and $s_2$, of $M_1$ and $M_2$ to $(N+1/2)$, we show that the exponential growth rate of the invariants coincides with the hyperbolic volume of the link complement equipped with certain (possibly incomplete) hyperbolic structure parametrized by $s_1$ and $s_2$. In the proof we figure out the correspondence between the critical point equations of the potential functions and the hyperbolic gluing equations of certain triangulations of the link complements. Furthermore, we discover a connection between the potential function, the theory of angle structures and the covolume function. As a corollary, we prove the volume conjecture for the Turaev-Viro invariants for all Whitehead chains complements. 

\end{abstract}

\section{Introduction}

\subsection{Overview of volume conjectures}
The volume conjecture of the colored Jones polynomials, discovered by R. Kashaev, H. Murakami and J. Murakami, suggests that the $N$-th colored Jones polynomials of a hyperbolic knot evaluated at the $N$-th root of unity $t=e^{\frac{2\pi i}{N}}$ grow exponentially with exponential growth rate equal to the hyperbolic volume of the knot complement.

\begin{conjecture}\label{cvc}\cite{K97,MM01} Let $K$ be a hyperbolic knot and $J'_{N}(K;t)$ be the normalized $N$-th colored Jones polynomials of $K$ evaluated at $t$ such that $J'_N(unknot,t)=1$ for all $N\in \NN$. We have
\begin{align*}
\lim_{N \to \infty} \frac{2\pi}{N} \log|J'_{N}(K;e^{\frac{2\pi i}{N}})| = \Vol(\SS^{3} \backslash K),
\end{align*}
where $\operatorname{Vol}(\SS^{3} \backslash K)$ is the hyperbolic volume of the knot complement in $\SS^3$.
\end{conjecture}

Later, H. Murakami and J. Murakami extended the above volume conjecture to non-hyperbolic links and suggested that the exponential growth rate captures the simplicial volume of the link complement.

\begin{conjecture}\label{cvcMM}\cite{MM01} Let $L$ be a link and $J'_{\vec{N}}(L;t)$ be the $\vec{N}$-th normalized colored Jones polynomials of $L$ evaluated at $t$, where $\vec{N} = (N,\dots, N)$. We have
\begin{align*}
\lim_{N \to \infty} \frac{2\pi}{N} \log|J'_{\vec{N}}(L;e^{\frac{2\pi i}{N}})| = v_3||\SS^3 \backslash L||,
\end{align*}
where $||\SS^3 \backslash L||$ is the simplicial volume of the link complement.
\end{conjecture}

In 2015, Q. Chen and T. Yang discovered a version of volume conjecture for the Turaev-Viro invariants at a primitive $2r$-th root of unity, where $r$ is odd. The conjecture can be stated as follows.

\begin{conjecture}\label{vctv}\cite{CY15}
For every hyperbolic 3-manifold $M$ with finite volume, we have
\begin{align*}
\lim_{\substack{ r\to \infty \\r \text{ odd}}}\frac{2\pi}{r}\log \left(TV_{r}(M,e^{\frac{2\pi i}{r}})\right) = \Vol(M)
\end{align*}
\end{conjecture}

It turns out when the $3$-manifold is a link complement, the Turaev-Viro invariant is related to the unnormalized colored Jones polynomials of the link evaluated at certain $(N+1/2)$-th root of unity as follows.

\begin{theorem}\label{relationship} \cite{DKY17}.
Let L be a link in $\SS^3$ with n components. Then given an odd integer $r=2N+1 \geq 3$, we have
\begin{align*}
TV_{r}\left(\SS^3 \backslash L, e^{\frac{2\pi i}{r}}\right) = 2^{n-1}\lt( \frac{2\sin(\frac{2\pi}{r})}{\sqrt{r}}\rt)^{2} \sum_{1\leq \vec{M} \leq \frac{r-1}{2}}\left|{J}_{\vec{M}}\left(L,e^{\frac{2\pi i }{N+\frac{1}{2}}}\right)\right|^2 .
\end{align*}
Here, $J_{\vec{M}}(L,t)$ is the unnormalized colored Jones polynomials of the link $L$ with
$$J_N(U,t) = \frac{t^{\frac{N}{2}} - t^{-\frac{N}{2}}}{t^{\frac{1}{2}}-t^{-\frac{1}{2}}}$$
for the unknot $U$ and any $N \in \NN$.
\end{theorem}

Therefore, by studying the asymptotics of the colored Jones polynomials of links at this $(N+1/2)$-th root of unity, one can relate the volume conjecture of the colored Jones polynomials of a link to the volume conjecture of the Turaev-Viro invariants of its complement. In particular, in \cite{DKY17}, the volume conjecture for the Turaev-Viro invariants is extended to the complement of non-hyperbolic links and this generalization is proved for all knots with zero simplicial volume. 

\begin{conjecture}\label{tvGM}\cite{DKY17}
For every link $L$ in $\SS^3$, we have
$$  \lim_{\substack{ r\to \infty \\r \text{ odd}}} \frac{2\pi }{r} \log(TV_r(\SS^3 \backslash L, e^{\frac{2\pi i}{r}})) = v_3 ||\SS^3 \backslash L||,$$
where $v_3$ is the volume of an ideal regular tetrahedron.
\end{conjecture}

Besides, in \cite{DKY17}, R. Detcherry, E. Kalfagianni and T. Yang ask the following question:
\begin{conjecture}\label{sVC}(Question 1.7 in \cite{DKY17}) Is it true that for any hyperbolic link $L$ in $\SS^3$, we have
$$ \lim_{N\to \infty} \frac{2\pi}{N+\frac{1}{2}} \log|J_{\vec{N}}(L,e^{\frac{2\pi i}{N+\frac{1}{2}}}) | = \Vol(\SS^3 \backslash L) \quad ?$$
\end{conjecture}
The answer is affirmative when $L$ is the figure eight knot. In fact, for the figure eight knot we can say more. In \cite{WA17}, Thomas Au and the author consider the $M$-th colored Jones polynomials of the figure eight knot evaluated at the $(N+1/2)$-th root of unity $t=e^{\frac{2\pi i}{N+1/2}}$. Here we regard $M$ as a sequence in $N$ and define the limiting ratio 
$$ s = \lim_{N \to \infty} \frac{M}{N+1/2}.$$
As a direct consequence of Theorem 7 (ii) in \cite{WA17}, we have the following result.
\begin{theorem}\cite{WA17}\label{WAthm} For the figure eight knot $4_1$, there exists $\delta > 0 $ such that for any $1-\delta < s \leq 1$, we have
$$ \lim_{N \to \infty} \frac{2\pi }{N+\frac{1}{2}}\log|J_M(4_1, e^{\frac{2\pi i}{N+\frac{1}{2}}}) | =  \Vol\lt(\SS^3 \backslash 4_1, u = 2\pi i (1 - s)\rt), $$ 
where $ \Vol\lt(\SS^3 \backslash 4_1, u = 2\pi i (1 - s)\rt)$ is the volume of $\SS^3 \backslash 4_1$ equipped with the (possibly incomplete) hyperbolic structure such that the logarithm of the holonomy of the meridian is given by $u = 2\pi i (1-s)$.
\end{theorem}

An important remark is that Theorerm~\ref{WAthm} is true for the root of unity $t=e^{\frac{2\pi i}{N+1/2}}$ but false for the original root $t=e^{\frac{2\pi i}{N}}$
\footnote{
This type of generalized volume conjecture has been studied in several literatures (e.g. \cite{GM08}, \cite{MY07}, \cite{HM13}). Unfortunately, in the original setting, i.e. $t=e^{\frac{2\pi i}{N}}$, this generalization fails even for knots (see Section 4 in \cite{C07}). Besides, the original volume conjecture is not true for split links due to the choice of normalization. Even if the normalization problem is fixed, in \cite{V08} R. van der Veen shows that the conjecture is false for the Whitehead chains with more than one belt.

On the other hand, if we consider the root $q=e^{\frac{2\pi i}{N+1/2}}$ rather than the original root, the above problems disappear and it gives a hope to have a generalized volume conjecture to hyperbolic links. Note that the difference between these two roots of unities has also been discussed in \cite{WA17}.}. As a generalization, we propose the generalized volume conjecture for hyperbolic
links at $t=e^{\frac{2\pi i}{N+1/2}}$ as follows (c.f. Question 4.6 and Theorem 4.8 in \cite{C07}).

\begin{conjecture}\label{GVCHL} (Generalized volume conjecture for hyperbolic links) Let $L$ be a hyperbolic link with $n$ components. Let $M_1, M_2, \dots , M_n$ be sequences of positive integers in $N$. Let $\ds s_i = \lim_{N\to\infty} \frac{M_i}{N+1/2}$. Then there exists $\delta_L > 0$ such that whenever the limiting ratio $s_i \in (1 - \delta_L, 1]$, we have
\begin{align*}
\lim_{N\to \infty} \frac{2\pi }{N+\frac{1}{2}} \log| J_{M_1,\dots,M_n}(L,e^{\frac{2\pi i}{N+\frac{1}{2}}})| = \Vol(\SS^3\backslash L, u_i = 2\pi i (1-s_i)),
\end{align*}
where $\Vol(\SS^3\backslash L, u_i = 2\pi i (1-s_i))$ is the volume of $\SS^3 \backslash L$ equipped with the hyperbolic structure such that the logarithm of the holonomy of the meridian around the $i$-th component is given by $u_i = 2\pi i(1 - s_i)$.
\end{conjecture}

\subsection{Strategy to prove the volume conjecture}\label{strate}
In this subsection, we briefly discuss our approach to study the volume conjecture. A standard way to prove volume conjecture involves three steps. Roughly speaking, first of all, given the explicit formula of the colored Jones polynomials of a link $L$ evaluated at the $(N+1/2)$-th root of unity, we convert it into an integral of the form 
$$\int_D f(z_1,z_2,\dots,z_n)\exp\lt( \lt(N+\frac{1}{2}\rt) \Phi_L(z_1,z_2,\dots,z_n) \rt) dz_1dz_2 \dots dz_n, $$
where $(z_1,z_2,\dots,z_n)\in \CC^n$ for some $n$, $f(z_1,z_2,\dots,z_n)$ is a holomorphic function and $\Phi_L(z_1,z_2,\dots,z_n)$ is called the potential function of the link $L$. After that, we apply the saddle point approximation, which says that the large $N$ behavior of the integral is determined by the critical value of the potential function at certain non-degenerate critical point. 

The last step is to show that the real part of the critical value gives us the hyperbolic volume of the link complement. Recall that in order to find the critical point of the potential function, we need to solve the critical point equations 
$$\frac{\partial}{\partial z_i}\Phi_L(z_1,z_2,\dots,z_n)=0 \quad\text{ for $i=1,2,\dots, n$}$$
It turns out that in many situations, the critical point equations of the potential function coincide with the hyperbolic gluing equations (e.g. edges equations, surgery equations) of certain triangulation of the link complement with shapes parameters parametrized by $z_1,z_2,\dots,z_n$. Finally, we need to argue that the critical value of the potential function gives the hyperbolic volume of the link complement.

This approach has been used to study the volume conjecture of the colored Jones polynomials for the figure eight knot \cite{HM13}, the $5_2$ knot \cite{O52}, and knots with 6 or 7 crossings \cite{O6,O7}. Moreover, the correspondence between the critical point equations of the potential function to the edges equations of certain triangulation has been studied for the twist knots \cite{CMY09} and two bridge knots \cite{KO05}.

Once the original volume conjecture is proved in this way, the generalized volume conjecture can be proved by the `continuity argument'. The continuity argument has been used in \cite{WA17}. The key point is that all the technical assumptions in the above analysis depend continuously on the potential function $\Phi_L(z_1,z_2,\dots,z_n)$. In particular, if we change the colors $M_i$, we obtain a family of potential functions which depends smoothly on the parameters $s_i$. Therefore, as long as $s_i$'s are sufficiently close to $1$, the above analysis works and the result follows.

\begin{remark}
It turns out that for the colored Jones polynomials of the Whitehead link and the Whitehead chains, we obtain two potential functions, and both of them give the hyperbolic volume of the link complement at certain critical points. See Section~\ref{formula} for more details.
\end{remark}

\subsection{Main results}

The main goal of this paper is to study the asymptotics of the colored Jones polynomials of the Whitehead chains defined by R. van der Veen in \cite{V08}.

For any $a\in \ZZ$, $b\in \NN$, $c,d\in\NN\cup\{0\}$ with $c+d\geq 1$, the Whitehead chain $W_{a,b,c,d}$ is obtained by stacking $a$ full twists, $b$ belts, $c$ clasps and $d$ mirror clasps and then taking the closure (Figure~\ref{fullpicture}). All the Whitehead chains $W_{a,1,c,d}$ are hyperbolic with
$$ \Vol(\SS^3\backslash W_{a,1,c,d}) = (c+d) v_8 ,$$
where $v_8$ is the volume of an ideal regular octahedron.

For $b>1$, the Whitehead chains $W_{a,b,c,d}$ is not hyperbolic but it contains a hypebolic piece in the JSJ decomposition with 
$$ v_3 ||\SS^3\backslash W_{a,b,c,d}|| = (c+d) v_8 $$
          \begin{figure}[H]
          \centering
              \includegraphics[width=0.8\linewidth]{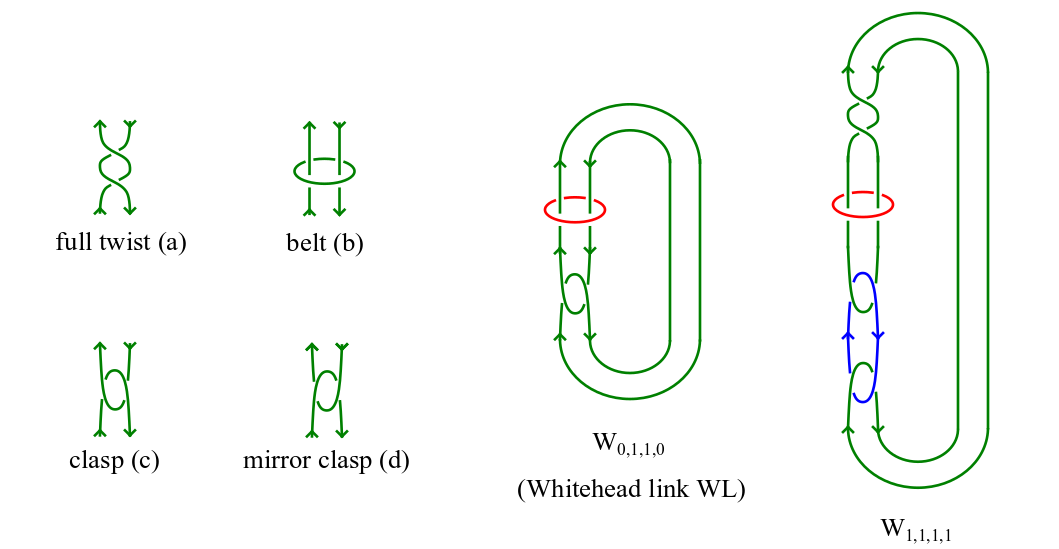}
              \caption{Left: the positive full twist, belt, clasp and mirror clasp tangles, when $a<0$, we twist the tangles in the opposite direction. Right: the links $W_{0,1,1,0}$ (also denoted as $WL$) and $W_{1,1,1,1}$}\label{fullpicture}
          \end{figure}

\subsubsection{Asymptotics of the quantum invariants of the Whitehead link}
Let $M_1$ and $M_2$ be sequences of integers in $N$ with limiting ratios
$$ s_i = \lim_{N \to \infty}\frac{M_i}{N+\frac{1}{2}}$$
We use $WL$ to denote the Whitehead link and $J_{M_1,M_2}(WL,t)$ to denote the colored Jones polynomials of the Whitehead link with the belt colored by $M_1$ and the clasp colored by $M_2$. In the first part of this paper, we study  the asymptotic expansion formula for $J_{M_1,M_2}(WL,e^{\frac{2\pi i}{N+\frac{1}{2}}})$. Our result can be stated as follows

\begin{theorem}\label{mainthmWL} For the $(M_1,M_2)$-th colored Jones polynomials of the Whitehead link evaluated at $t=e^{\frac{2\pi i}{N+\frac{1}{2}}}$,
\begin{enumerate}
\item there exist two functions $\Phi^{\pm(s_1,s_2)}(z_1,z_2)$, a triangulation of $\SS^3 \backslash WL$ and an assignment of shape parameters such that
\begin{enumerate}
\item the critical point equations 
$$\frac{\partial}{\partial z_1}\Phi^{\pm(s_1,s_2)}(z_1,z_2)
=\frac{\partial}{\partial z_2}\Phi^{\pm(s_1,s_2)}(z_1,z_2)
=0$$
coincide with the hyperbolic gluing equations for this particular triangulation.
\item there exist two families of critical points $\mathbf{z}^{\pm(s_1, s_2)}$ which holomorphically depend on $s_1,s_2$ such that 
$$  \Phi^{+(s_1,s_2)}(\mathbf{z}^{+(z_1,z_2)})
= \Phi^{-(s_1,s_2)}(\mathbf{z}^{-(z_1,z_2)}) $$
and
$$ \Re \Phi^{\pm(s_1, s_2)}(\mathbf{z}^{\pm(s_1, s_2)}) = \Vol\lt(\SS^3 \backslash WL, u_1 = 2\pi i (1 - s_1), v_2=4\pi i (1 - s_2)\rt) ,$$
where $\Vol\lt(\SS^3 \backslash WL, u_1 = 2\pi i (1 - s_1), v_2=4\pi i (1 - s_2)\rt)  $ is the hyperbolic volume of the Whitehead link complement equipped with the incomplete hyperbolic structure such that the logarithm of the holonomy of the meridian of the belt and the logarithm of the holonomy of the longitude of the clasp are $u_1 = 2\pi i (1 - s_1)$ and $v_2=4\pi i (1 - s_2)$ respectively.
\end{enumerate}
\item We have
\begin{align*}
\lim_{N\to \infty} \frac{2\pi}{N+\frac{1}{2}} \log \lt| J_{M_1,M_2}(WL,e^{\frac{2\pi i}{N+\frac{1}{2}}}) \rt| 
&= \Re \Phi^{\pm(s_1, s_2)}(\mathbf{z}^{\pm(s_1, s_2)}) \\
&= \Vol\lt(\SS^3 \backslash WL, u_1 = 2\pi i (1 - s_1), v_2=4\pi i (1 - s_2)\rt)
\end{align*}
\end{enumerate}
\end{theorem}

As a corollary, by exchanging the components of the Whitehead link, we obtain the following symmetry for the volume function of the Whitehead link complement.

\begin{corollary}\label{WLVolsym}
We have
\begin{align}\label{Volsym1}
\Vol(\SS^3\backslash WL, u_1 = 2\pi i(1-s_1), v_2= 4\pi i(1-s_2))
= \Vol(\SS^3\backslash WL, v_1= 4\pi i(1-s_1), u_2 = 2\pi i(1-s_2))
\end{align}
In particular,
\begin{align}\label{Volsym2}
\Vol(\SS^3\backslash WL, u_1 = 2\pi i(1-s_1), u_2=0)
= \Vol(\SS^3\backslash WL, v_1= 4\pi i(1-s_1), u_2 = 0)
\end{align}
\end{corollary}

Theorem~\ref{mainthmWL} and Corollary~\ref{WLVolsym} give a positive answer to Conjecture~\ref{GVCHL} for the sequences of colors $M_1,M_2$ with either $s_1=1$ or $s_2=1$. In the general case, it is not clear to the author whether Theorem~\ref{mainthmWL} could give a positive (or negative) answer to Conjecture~\ref{GVCHL}. 

%%%%%%%%%%%%%%%%%%%%%%
\iffalse
In \cite{WA17}, the authors descirbe a procedure to compute the asymptotic expansion formula (AEF) for the Turaev-Viro invariants of the link complement from that of the colored Jones polynomials of the link. Note that the complements of all $a$-twisted Whitehead links $WL(a)$ are homeomorphic. In particular, in order to study their Turaev-Viro invariants, it suffices to study that of the Whitehead link. Let 
$$H(s_1,s_2) = \frac{1}{2\pi}\Vol\lt(\SS^3 \backslash WL, u_i = 2\pi i (1 - s_i)\rt)$$ By using Theorem~\ref{mainthmWL}, we have the following result.
\begin{theorem}Conjecture~\ref{vctv} is true for all twisted Whitehead link complements. Furthermore, the AEF of $TV_r(\SS^3 \backslash WL(a), e^{\frac{2\pi i}{r}})$ is given by
\begin{align*}
TV_{r}\left(\SS^3 \backslash WL(a), e^{\frac{2\pi i}{r}}\right) \stackrel[N \to \infty]{\sim}{ }   \frac{(N+\frac{1}{2})\pi}{\sqrt{2}\sqrt{\det \Hess H(1,1)}} \exp\lt(\frac{r}{2\pi}\times\Vol(WL(a))\rt),
\end{align*}
\end{theorem}
\fi
%%%%%%%%%%%%%%%%%%%%%

\subsubsection{Asymptotic of the quantum invariants of the Whitehead chains}
In the second part of this paper, we generalize the previous results on the Whitehead link to the Whitehead chains. 

\begin{theorem}\label{CJWNN} 
We have
\begin{align*}
\lim_{N\to \infty} \frac{2\pi}{N+\frac{1}{2}} \log| J_{\vec{N}} (W_{a,1,c,d}, e^{\frac{2\pi i}{N+\frac{1}{2}}})| 
= \Vol(\SS^3 \backslash W_{a,1,c,d})
\end{align*}
\end{theorem}

As a consequence, we can show that
\begin{corollary}\label{TVW}For any $a\in \ZZ$, $b\geq 1$, $c,d \in \NN$ with $c+d\geq 1$, 
Conjectures~\ref{tvGM} is true for $TV_r(\SS^3\backslash W_{a,b,c,d})$.
\end{corollary}

Moreover, if we colored the belt by $M_1$ and all the clasps by $M_2$ (see Figure~\ref{W0120} for an example), we have the following result.

\begin{theorem}\label{mainthmWa1c0} 
There exists some $\delta>0$ such that for any $1-\delta < s_1,s_2 < 1$, we have
\begin{align*}
&\quad\lim_{N\to \infty} \frac{2\pi}{N+\frac{1}{2}} \log| J_{M_1, M_2} (W_{0,1,c,0}, e^{\frac{2\pi i}{N+\frac{1}{2}}})| \\
&= \Vol(\SS^3 \backslash W_{0,1,c,0} , u_1=2\pi i (1-s_1), v_2 = v_3 = \dots v_{c+1} = 4\pi i (1 - s_2))
\end{align*}
\end{theorem}

          \begin{figure}[H]
          \centering
              \includegraphics[width=0.65\linewidth]{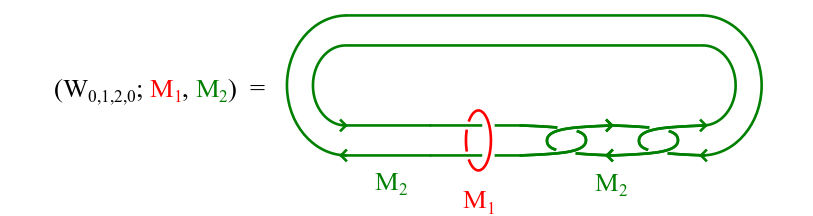}
              \caption{the link $W_{0,1,2,0}$ with belt colored by $M_1$ and all other components colored by $M_2$}\label{W0120}
          \end{figure}
          
Note that Theorem~\ref{mainthmWa1c0} gives a positive answer to Conjecture \ref{GVCHL} for the sequence of colors $(M_1,M_2,\dots,M_2)$ with $s_2=1$.

\subsubsection{Differential formula for the potential function and relation to hyperbolic geometry}
From the discussion in subsection~\ref{strate}, we have already seen the importance of the potential function. Thus, it is necessary to have a better understanding about how the geometric information of a link complement is encoded in the potential function of the link. In the study of the asymptotics of the colored Jones polynomials, we discover the following formula satisfied by the potential function of the Whitehead link. \newpage

\begin{lemma}\label{dformula}
Write $z_k = x_k + i y_k$ for $k=1,2$. The real part of the potential function satisfies the following differential equation: 
\begin{align} \label{DF}
\Re \Phi^{+(s_1,s_2)}(z_1,z_2) = \frac{1}{2\pi} V^{+(s_1,s_2)}(z_1,z_2) + \sum_{k=1}^2 y_k \frac{\partial}{\partial y_k} \Re \Phi^{s_1,s_2}(z_1,z_2),
\end{align}
where 
\begin{align*}
V^{+(s_1,s_2)}(z_1,z_2) 
&= D(e^{2\pi i(s_2-1) - 2\pi i z_1 - 2\pi i z_2}) - D(e^{2\pi i (s_2-1) - 2\pi i z_2}) + D(e^{2\pi i z_2}) \\
&\qquad - D(e^{2\pi i z_1 + 2\pi i z_2}) + D(e^{2\pi i z_1}) 
\end{align*}
and $D(z)$ is the Bloch-Wigner function.
\end{lemma}

This formula has two immediate consequences. First, from Theorem~\ref{mainthmWL})1)a), there exists a family of critical point $\mathbf{z}^{+(s_1,s_2)}$ for the potential function $ \Phi^{+(s_1,s_2)}(z_1,z_2) $ such that the critical point equations coincide with the hyperbolic gluing equations for a particular triangulation of the Whitehead link complement. In particular, at this crtical point, by definition we have 
$$  \frac{\partial}{\partial y_k} \Re \Phi^{+(s_1,s_2)}(\mathbf{z}^{+(s_1,s_2)}) = 0
 \text{ for $k=1,2$.}$$
From this, using lemma above, we can see that the real part of the critical value of the potential function is the sum of the volume of the ideal tetrahedra which satisfy the hyperbolic gluing equations, i.e. the hyperbolic volume of the link complement. 

Next, the formula relates the potential function to the theory of angle structures. Note that every points on the subset 
$$\mathfrak{A} = \lt\{(z_1,z_2) \mid \im \frac{d}{dz_i} \Phi^{+(s_1,s_2)}(z_1,z_2) =  \frac{\partial}{\partial y_k} \Re \Phi^{\pm(s_1,s_2)}(z_1,z_2) = 0 \text{ for $k=1,2$}\rt\}$$ 
satisfies the property that the sum of the angles around the vertices are $2\pi$. Therefore,  $\mathfrak{A}$ is  a subset of the space of angle structures and the potential function restricted on $\mathfrak{A}$ is exactly the volume function. In particular, by the theorem of angle structures, we know that the maximum point on $\mathfrak{A}$ is exactly the hyperbolic volume of the link complement. 

Moreover, it is interesting to compare the differential formula~(\ref{DF}) with the covolume function discussed in \cite{LY18}. Recall that every hyper-ideal tetrahedron is determined (up to isometry) by six dihedral angles $\{a_{ij}\}_{1\leq i <j \leq 4}$, where $a_{ij}$ is the dihedral angle between the faces $i$ and $j$. Let $l_{ij}$ be the length of the edge between the face $i$ and the face $j$. Then $a_{ij}$ and $l_{ij}$ uniquely determine each other. Moreover, the volume function $vol$ satisfies the Sch$\ddot{\text{a}}$lfi's formula:
$$ \frac{\partial vol}{\partial a_{ij}} = - \frac{l_{ij}}{2}$$
The Legendre transform of the volume function is called the covolume function $cov$, which is given by \cite{LY18}
$$ cov (l_{ij}) = vol(a_{ij}) + \sum_{i<j } l_{ij} \cdot \frac{a_{ij}}{2} $$
Note that by the Sch$\ddot{\text{a}}$lfi's formula,
$$ \frac{\partial cov}{\partial l_{ij}} = \sum_{i<j}  \frac{\partial vol}{\partial a_{ij}}\frac{\partial a_{ij}}{\partial l_{ij}} + \sum_{i<j} \frac{l_{ij}}{2}\frac{\partial a_{ij}}{\partial l_{ij}} + \frac{a_{ij}}{2} = \frac{a_{ij}}{2}  $$
Thus, we have
\begin{align}\label{cov}
cov(l_{ij}) = vol(a_{ij}) + \sum_{i<j} l_{ij}\frac{\partial cov}{\partial l_{ij}} 
\end{align}
Now, given a triangulation $\{\Delta_1,\dots, \Delta_n\}$ of a 3-manifold $M$ with edge lengths $\{l^1_{ij},\dots,l^n_{ij}\}$, where we identify $l^a_{pq}$ with $l^b_{rs}$ whenever these two edges are glued together, we can define the covolume function $cov_M$ of $M$ to be
\begin{align}\label{cov2}
cov_M (l^1_{ij}, \dots, l^n_{ij}) = \sum_{k=1}^n cov_k (l_{ij}) = \sum_{k=1}^n vol_k(a_{ij}) + \sum_{i,j}  l_{ij} A_{ij}
\end{align}
where $A_{ij} = \frac{\partial cov_M}{\partial l_{ij}}$ is the sum of the dihedral angles around the edge $l_{ij}$. Equations (\ref{DF}), (\ref{cov}) and (\ref{cov2}) suggest that we should think of $y_k$ as some sort of `length'. We hope that this formula will help us to understand the theory of potential function.

The final remark for this equation is that the proof of Lemma~\ref{DF} is based on some basic properties of the dilogarithm function and the Lobachevsky function. In particular, by using the argument, one can show that the potential functions discussed in \cite{O52,O6,O7}, and the potential functions in \cite{KO05,CMY09} after the reparametrization $z \mapsto e^{2\pi i z}$, also satisfy the same type of differential formula. 

\subsection{Outline}
In Section~\ref{formula}, we first compute the potential functions of $J_{M_1,M_2}(WL,e^{\frac{2\pi i}{N+\frac{1}{2}}})$. By using the Poisson summation formula and the saddle point approximation, we compute its asymptotic expansion formula. Then we study how the critical point equations and the critical values are related to the geometry of the Whitehead link complement. Besides, we prove the differential formula for the potential function. Altogether, we can prove Theorem~\ref{mainthmWL}. In Section~\ref{gen}, we compute the AEF for $J_{\vec{N}}(W_{a,1,c,d},e^{\frac{2\pi i}{N+\frac{1}{2}}})$ and generalize previous results to the Whitehead chains with $1$ belt and $c$ clasps. In Section~\ref{TV}, by using the results in the previous section, we prove Corollary~\ref{TVW}.

\subsection{Acknowledgements}
The author would like to thank Tian Yang for the valuable comments and suggestions to this work. 

\section{Asymptotics of $J_{M_1,M_2}(WL,e^{\frac{2\pi i}{N+\frac{1}{2}}})$}\label{formula}
\subsection{Potential functions of $J_{M_1,M_2}(WL,e^{\frac{2\pi i}{N+\frac{1}{2}}})$}
In this section, we compute the potential function of the $(M_1,M_2)$-th unnormalized colored Jones polynomials for the Whitehead link $WL$, where $M_1,M_2$ are sequences of positive integers in $N$. First of all, the formula of the $(M_1,M_2)$-th unnormalized colored Jones polynomials of the Whitehead link is given by \cite{HZ07}
$$ J_{M_1,M_2} (WL, t) = t^{\frac{M_2^2-1}{2}} (t^{\frac{1}{2}}-t^{-\frac{1}{2}})^{-1}\sum_{n=0}^{M_2-1} (t^{\frac{M_1(2n+1)}{2}}-t^{-\frac{M_1(2n+1)}{2}})  \cdot \tilde{C}(n,t; M_2) ,$$
where 
$$\tilde{C}(n,t; M_2) = t^{\frac{M_2(M_2-1)}{2}} \sum_{l = 0}^{M_2-1-n} t^{-M_2(l+n)} \prod_{j=1}^{n} \frac{(1-t^{M_2-l-j})(1-t^{l+j})}{1-t^j}.$$
Put $t = e^{\frac{2\pi i}{N+\frac{1}{2}}}$. By direct computation, we have
\begin{align}\label{t0}
 (t^{\frac{1}{2}}-t^{-\frac{1}{2}})^{-1} = \frac{1}{2i\sin\lt(\frac{\pi}{N+\frac{1}{2}}\rt)}
\end{align}
\begin{align}
 t^{\frac{M_1(2n+1)}{2}}-t^{-\frac{M_1(2n+1)}{2}} 
&= t^{\frac{(N+\frac{1}{2})(2n+1)}{2} + \frac{(M_1 - (N+\frac{1}{2}))(2n+1)}{2}} - t^{-\frac{(N+\frac{1}{2})(2n+1)}{2} - \frac{(M_1 - (N+\frac{1}{2}))(2n+1)}{2}}  \notag \\
&= -\lt[ e^{2\pi i(\frac{M_1}{N+\frac{1}{2}} -1) (n+\frac{1}{2})} - e^{-2\pi i (\frac{M_1}{N+\frac{1}{2}} -1)(n+\frac{1}{2})} \rt]  \notag \\
&= -\lt[ e^{\pi i\lt(\frac{M_1}{N+\frac{1}{2}}-1\rt)}e^{2\pi i\lt(\frac{M_1}{N+\frac{1}{2}} -1\rt) (n)} - e^{-\pi i\lt(\frac{M_1}{N+\frac{1}{2}}-1\rt)}e^{-2\pi i \lt(\frac{M_1}{N+\frac{1}{2}} -1\rt)(n)} \rt]  \notag \\
&= (-1)^{M_1 - N}i\lt[ e^{\pi i\lt(\frac{M_1}{N+\frac{1}{2}}-1\rt)}e^{\frac{N+\frac{1}{2}}{2\pi i}\lt[(2\pi i)\lt(\frac{M_1}{N+\frac{1}{2}} -1\rt)\rt] \lt[2\pi i \lt(\frac{n}{N+\frac{1}{2}}-\frac{1}{2}\rt)\rt]} \rt. \notag\\
&\qquad\qquad\qquad\qquad\qquad \lt. + e^{-\pi i\lt(\frac{M_1}{N+\frac{1}{2}}-1\rt)}e^{-\frac{N+\frac{1}{2}}{2\pi i}\lt[(2\pi i)\lt(\frac{M_1}{N+\frac{1}{2}} -1\rt)\rt] \lt[2\pi i \lt(\frac{n}{N+\frac{1}{2}}-\frac{1}{2}\rt)\rt]} \rt]
\label{t1}
\end{align}
Besides,
\begin{align}
t^{\frac{M_2^2 - 1}{2} + \frac{M_2(M_2-1)}{2}}
&= e^{\frac{2\pi i}{N+\frac{1}{2}}(M_2^2 - \frac{M_2}{2} - \frac{1}{2})} 
\\\notag\\
t^{-M_2(l+n)} 
&=  e^{-2\pi i (\frac{M_2}{N+\frac{1}{2}}-1) (l+n)} \notag \\
&= e^{- \frac{N+\frac{1}{2}}{2\pi i} \lt[ 2\pi i \lt(\frac{M_2}{N+\frac{1}{2}} -1 \rt) \lt[ 2\pi i \lt(\frac{l}{N+\frac{1}{2}}\rt) + 2\pi i \lt( \frac{n}{N+\frac{1}{2}}\rt) \rt] \rt]} \label{t5}
\end{align}
Next, for any odd integer $r=2N+1\geq 3$, we consider the following version of the quantum dilogarithm function $\varphi_r(z)$ (see \cite{F95, FKV01}, or Section 2.3 in \cite{WY20} for a review of its properties) defined by 
\begin{align*}
\varphi_r (z) = \frac{4\pi i}{r}\int_\Omega \frac{e^{(2z - \pi)x} }{4x \sinh(\pi x) \sinh (\frac{2\pi x}{r} )} dx,
\end{align*}
where 
\begin{align*}
\Omega = (-\infty, - \epsilon] \cup \{ z \in \CC \mid |z| = \epsilon, \im z >0\} \cup [ \epsilon, \infty)
\end{align*}
for some $\epsilon \in (0,1)$ and 
\begin{align*}
z \in \lt\{ z \in \CC \lt\vert - \frac{\pi}{2N+1} < \Re z < \pi + \frac{\pi}{2N+1}\rt. \rt\}
\end{align*}
For any $z \in \CC$ with $0 < \Re z < \pi$, the quantum dilogarithm function satisfies the functional equation (Lemma 2.1)1) in \cite{WY20})
\begin{align}
1 - e^{2i z} = \exp \lt(\frac{N+\frac{1}{2}}{2\pi i}\lt( \varphi_r \lt( z - \frac{\pi }{2N+1} \rt) - \varphi_r \lt( z + \frac{\pi }{2N+1} \rt) \rt)\rt)
\end{align}
Furthermore, let $\Li: \CC\backslash (1,\infty) \to \CC$ be the dilogartihm function defined by
$$ \Li(z) = - \int_0^z \frac{\log(1-u)}{u} du $$
For any $z$ with $0<\Re z < \pi$, the quantum dilogarithm function satisfies (Lemma 2.3 in \cite{WY20})
\begin{align}
\varphi_r (z) = \Li(e^{2iz}) + \frac{2\pi^2 e^{2iz}}{3(1-e^{2iz})} \frac{1}{(2N+1)^2} + O\lt(\frac{1}{(N+\frac{1}{2})^3}\rt) \label{QtoC},
\end{align}
and as $N \to \infty$, $\varphi_r(z)$ uniformly converges to $\Li(e^{2iz})$ on any compact subset of $\{z \in \CC \mid 0 < \Re z < \pi \}$.

Using the functional equation of the quantum dilogarithm function, we have
\begin{align}
 \prod_{j=1}^{n} (1-t^{M_2-l-j}) 
=&\prod_{j=1}^n \lt(1 - e^{2\pi i \lt(\frac{M_2}{N+\frac{1}{2}} - \frac{l}{N+\frac{1}{2}} - \frac{j}{N+\frac{1}{2}}\rt)}\rt)\notag \\
=&\exp\lt(\frac{N+\frac{1}{2}}{2\pi i} \sum_{j=1}^n \lt( \varphi_r \lt(\frac{M_2 \pi }{N+\frac{1}{2}} - \frac{l\pi}{N+\frac{1}{2}} - \frac{(j+\frac{1}{2})\pi}{N+\frac{1}{2}}\rt) \rt. \rt.\notag\\
&\lt. \lt.\qquad - \varphi_r\lt( \frac{M_2\pi}{N+\frac{1}{2}} - \frac{l\pi}{N+\frac{1}{2}} - \frac{(j-\frac{1}{2})\pi}{N+\frac{1}{2}}\rt)  \rt) \rt) \notag\\
=& \exp \lt(\frac{N+\frac{1}{2}}{2\pi i} \lt(\varphi_r \lt( \frac{M_2\pi}{N+\frac{1}{2}} - \frac{l\pi}{N+\frac{1}{2}} - \frac{n\pi}{N+\frac{1}{2}} - \frac{\pi}{2N+1}\rt) \rt.\rt.\notag\\
&\lt.  \lt.\qquad - \varphi_r\lt(\frac{M_2\pi}{N+\frac{1}{2}} - \frac{l\pi}{N+\frac{1}{2}} - \frac{\pi}{2N+1}\rt)  \rt)  \rt)\label{t6}
\end{align}
\begin{align}
 \prod_{j=1}^{n} (1-t^{l+j})
=&\prod_{j=1}^n \lt(1 - e^{2\pi i \lt(\frac{l}{N+\frac{1}{2}} + \frac{j}{N+\frac{1}{2}}\rt)}\rt)\notag \\
=&\exp\lt(\frac{N+\frac{1}{2}}{2\pi i} \sum_{j=1}^n \lt( \varphi_r \lt(\frac{l\pi}{N+\frac{1}{2}} + \frac{(j-\frac{1}{2})\pi}{N+\frac{1}{2}}\rt)  - \varphi_r\lt( \frac{l\pi}{N+\frac{1}{2}} + \frac{(j+\frac{1}{2})\pi}{N+\frac{1}{2}}\rt)  \rt) \rt)\notag \\
=&\exp\lt(\frac{N+\frac{1}{2}}{2\pi i} \lt( \varphi_r \lt(\frac{l\pi}{N+\frac{1}{2}} + \frac{\pi}{2N+1}\rt)  - \varphi_r\lt( \frac{l\pi}{N+\frac{1}{2}} + \frac{n\pi}{N+\frac{1}{2}}+\frac{\pi}{2N+1}\rt)  \rt) \rt)\label{t7}
\\ \notag\\
 \prod_{j=1}^{n} (1-t^{j})
=&\prod_{j=1}^n \lt(1 - e^{2\pi i \lt(\frac{j}{N+\frac{1}{2}}\rt)}\rt) \notag\\
=&\exp\lt(\frac{N+\frac{1}{2}}{2\pi i} \lt(\sum_{j=1}^n \lt( \varphi_r \lt( \frac{(j-\frac{1}{2})\pi}{N+\frac{1}{2}}\rt)  - \varphi_r\lt( \frac{(j+\frac{1}{2})\pi}{N+\frac{1}{2}}\rt)  \rt) \rt)\rt)\notag \\
=&\exp\lt(\frac{N+\frac{1}{2}}{2\pi i} \lt( \varphi_r \lt(  \frac{\pi}{2N+1}\rt)  - \varphi_r\lt( \frac{n\pi}{N+\frac{1}{2}} + \frac{\pi}{2N+1}\rt)   \rt) \rt) \label{t8}
 \end{align}
 
Altogether, by (\ref{t0}) - (\ref{t8}), we have
\begin{align}
&\quad J_{M_1, M_2} (WL, e^{\frac{2\pi i}{N+\frac{1}{2}}}) \notag \\
&=  \frac{ (-1)^{(M_1 - N) }  e^{\frac{2\pi i}{N+\frac{1}{2}}(M_2^2 - \frac{M_2}{2} - \frac{1}{2})}}{2\sin(\frac{\pi}{N+\frac{1}{2}})} \notag \\
&\qquad \times \sum_{n=0}^{M_2-1} 
\lt[ e^{\pi i\lt(\frac{M_1}{N+\frac{1}{2}}-1\rt)}e^{\frac{N+\frac{1}{2}}{2\pi i}\lt[(2\pi i)\lt(\frac{M_1}{N+\frac{1}{2}} -1\rt)\rt] \lt[2\pi i \lt(\frac{n}{N+\frac{1}{2}}-\frac{1}{2}\rt)\rt]} \rt. \notag\\
&\qquad\qquad\qquad\qquad\qquad \lt. + e^{-\pi i\lt(\frac{M_1}{N+\frac{1}{2}}-1\rt)}e^{-\frac{N+\frac{1}{2}}{2\pi i}\lt[(2\pi i)\lt(\frac{M_1}{N+\frac{1}{2}} -1\rt)\rt] \lt[2\pi i \lt(\frac{n}{N+\frac{1}{2}}-\frac{1}{2}\rt)\rt]} \rt] \notag \\
&\qquad  \times \sum_{l=0}^{M_2-1-n}   e^{- \frac{N+\frac{1}{2}}{2\pi i} \lt[ 2\pi i \lt(\frac{M_2}{N+\frac{1}{2}} -1 \rt) \lt[ 2\pi i \lt(\frac{l}{N+\frac{1}{2}}\rt) + 2\pi i \lt( \frac{n}{N+\frac{1}{2}}\rt) \rt] \rt]} \notag\\ 
&\qquad \times \exp\lt(\frac{N+\frac{1}{2}}{2\pi i} \lt(\varphi_r \lt(\frac{M_2\pi}{N+\frac{1}{2}} - \frac{l\pi}{N+\frac{1}{2}} - \frac{n\pi}{N+\frac{1}{2}} - \frac{\pi}{2N+1}\rt)\rt. \rt. \notag\\
&\qquad\qquad\qquad  - \varphi_r\lt(\frac{M_2\pi}{N+\frac{1}{2}} - \frac{l\pi}{N+\frac{1}{2}} - \frac{\pi}{2N+1}\rt) \notag \\
&\qquad\qquad\qquad  + \varphi_r \lt(\frac{l\pi}{N+\frac{1}{2}} + \frac{\pi}{2N+1}\rt)  - \varphi_r\lt( \frac{l\pi}{N+\frac{1}{2}} + \frac{n\pi}{N+\frac{1}{2}} + \frac{\pi}{2N+1}\rt) \notag\\
&\lt.\lt.\qquad\qquad\qquad  - \varphi_r \lt(  \frac{\pi}{2N+1}\rt)  + \varphi_r\lt( \frac{n\pi}{N+\frac{1}{2}} + \frac{\pi}{N+\frac{1}{2}}\rt) \rt)\rt) 
\end{align}
Define the limiting ratio $s_1$ and $s_2$ by
$$ s_1 = \lim_{N \to \infty} \frac{M_1}{N+\frac{1}{2}} \quad\text{ and }\quad s_2 = \lim_{N \to \infty} \frac{M_1}{N+\frac{1}{2}},$$
Then we have
\begin{align}
J_{M_1, M_2} (WL, e^{\frac{2\pi i}{N+\frac{1}{2}}}) 
&=  \frac{ (-1)^{(M_1 - N) }  e^{\frac{2\pi i}{N+\frac{1}{2}}(M_2^2 - \frac{M_2}{2} - \frac{1}{2})}}{2\sin(\frac{\pi}{N+\frac{1}{2}})} \exp\lt(  -\varphi_r \lt(  \frac{\pi}{2N+1}\rt)\rt)
 (I_+ + I_-)
\end{align}
where
\begin{align}
I_\pm =
& e^{\pm\pi i \lt(\frac{M_2}{N+\frac{1}{2}} - 1\rt) }\sum_{n=0}^{M_2-1} \sum_{l=0}^{M_2-1-n} \exp\lt( \lt(N+\frac{1}{2}\rt)\Phi^{\pm(s_1, s_2)}_{M_1, M_2} \lt(\frac{n}{N+\frac{1}{2}},\frac{l}{N+\frac{1}{2}}\rt) \rt)
\end{align}
with
\begin{align}
\Phi^{\pm(s_1, s_2)}_{M_1, M_2} (z_1,z_2) 
&= \frac{1}{2\pi i}\lt\{ \pm \lt(2\pi i \lt(\frac{M_1}{N+\frac{1}{2}} - 1\rt)\rt) \lt(2\pi i \lt(z_1 -\frac{1}{2} \rt)\rt) \rt. \notag \\
&\qquad \lt.  - \lt(2\pi i \lt(\frac{M_2}{N+\frac{1}{2}} - 1\rt)\rt)(2\pi i z_1 + 2\pi i z_2) \rt. \notag \\
&\qquad  + \lt[\varphi_r \lt( \frac{M_2\pi}{N+\frac{1}{2}}  - \pi z_1 - \pi z_2 -\frac{\pi}{2N+1}\rt) \rt.  \notag\\
&\qquad  - \varphi_r\lt(  \frac{M_2\pi}{N+\frac{1}{2}} - \pi z_2- \frac{\pi}{2N+1}\rt) \notag \\
&\qquad  + \varphi_r \lt(\pi z_2 + \frac{\pi}{2N+1}\rt)  - \varphi_r\lt( \pi z_1 + \pi z_2 + \frac{\pi}{2N+1}\rt) \notag\\
&\lt.\lt.\qquad  + \varphi_r\lt( \pi z_1 + \frac{\pi}{2N+1}\rt) \rt] \rt\}
\end{align}
defined on $\lt\{ (z_1,z_2) \in \CC \mid 0 \leq \Re z_1 , \Re z_2 \leq \pi, \Re z_1 + \Re z_2 < \frac{M_2}{N+\frac{1}{2}} \rt\}$. Take $N\to \infty$, we obtain the potential functions 
\begin{align*}
\Phi^{\pm(s_1, s_2)} (z_1,z_2) 
&=  \frac{1}{2\pi i} \lt[ \pm \lt(2\pi i \lt( s_1 - 1\rt)\rt) \lt(2\pi i \lt(z_1 -\frac{1}{2} \rt)\rt)  - (2\pi i (s_2-1))(2\pi i z_1 + 2\pi i z_2)  \rt. \\
&\qquad + \Li\lt(e^{2\pi i (s_2-1) - 2\pi i z_1 - 2\pi i z_2}\rt) - \Li\lt( e^{2\pi i (s_2-1) - 2\pi i z_2 }\rt) \\
&\qquad \lt.+ \Li \lt(e^{2\pi i z_2}\rt) - \Li\lt(e^{2\pi iz_1+ 2\pi i z_2}\rt) + \Li\lt(e^{2\pi i z_1}\rt)  \rt]
\end{align*}
defined on $\lt\{ (z_1,z_2) \in \CC \mid \Re z_1 , \Re z_2 >0, \Re z_1 + \Re z_2 < s_2 \rt\}$.

The following lemma estimates the difference between $\Phi^{\pm(s_1, s_2)}_{M_1, M_2} (z_1,z_2) $ and $\Phi^{\pm\lt(\frac{M_1}{N+\frac{1}{2}}, \frac{M_2}{N+\frac{1}{2}}\rt)} (z_1,z_2) $.

\begin{lemma}\label{EM_2}
On any compact subset of $\lt\{ (z_1,z_2) \in \CC \mid \Re z_1 , \Re z_2 >0, \Re z_1 + \Re z_2 < s_2 \rt\}$,
\begin{align}
\Phi^{\pm(s_1,s_2)}_{M_1,M_2}(z_1, z_2 )
= \Phi^{\pm\lt(\frac{M_1}{N+\frac{1}{2}}, \frac{M_2}{N+\frac{1}{2}}\rt)} (z_1,z_2) +\frac{1}{N+\frac{1}{2}}E_{M_2}(z_1,z_2) + O\lt(\frac{1}{\lt(N+\frac{1}{2}\rt)^2}\rt),
\end{align}
where
\begin{align}
E_{M_2}(z_1,z_2) 
&= \frac{1}{2}\lt[ \log\lt(1-e^{2\pi i \lt(\frac{M_2}{N+\frac{1}{2}}-1\rt) - 2\pi i z_1 - 2\pi i z_2 }\rt)
- \log\lt(1-e^{2\pi i \lt(\frac{M_2}{N+\frac{1}{2}}-1\rt) - 2\pi i z_2 } \rt) \rt.\notag\\
&\lt.\qquad - \log\lt(1- e^{2\pi i z_2 } \rt)
+ \log\lt(1- e^{2\pi iz_1+ 2\pi i z_2} \rt) 
- \log\lt(1-  e^{2\pi i z_1 } \rt) \rt]
\end{align}
\end{lemma}

\begin{proof}
From (\ref{QtoC}), we have
\begin{align}
&\Phi^{+(s_1,s_2)}_{M_1,M_2}(z_1, z_2 ) \notag\\
&=  \frac{1}{2\pi i} \lt[ \pm \lt(2\pi i \lt( \frac{M_1}{N+\frac{1}{2}} - 1\rt)\rt) \lt(2\pi i \lt(z_1 -\frac{1}{2} \rt)\rt)  - \lt(2\pi i \lt(\frac{M_2}{N+\frac{1}{2}}-1\rt)\rt)(2\pi i z_1 + 2\pi i z_2)  \rt. \notag\\
&\qquad + \Li\lt(e^{2\pi i \lt(\frac{M_2}{N+\frac{1}{2}}-1\rt) - 2\pi i z_1 - 2\pi i z_2 - \frac{\pi i}{N+\frac{1}{2}}}\rt) - \Li\lt( e^{2\pi i \lt(\frac{M_2}{N+\frac{1}{2}}-1\rt) - 2\pi i z_2 - \frac{\pi i}{N+\frac{1}{2}}}\rt) \notag\\
&\qquad \lt.+ \Li \lt(e^{2\pi i z_2 + \frac{\pi i}{N+\frac{1}{2}}}\rt) - \Li\lt(e^{2\pi iz_1+ 2\pi i z_2 + \frac{\pi i}{N+\frac{1}{2}}}\rt) + \Li\lt(e^{2\pi i z_1 + \frac{\pi i}{N+\frac{1}{2}}}\rt)  \rt] + O\lt(\frac{1}{(N+\frac{1}{2})^2}\rt)
\end{align}
By considering the Talyor series expansion of $\Li\lt(e^{2\pi i \lt(\frac{M_2}{N+\frac{1}{2}}-1\rt) - 2\pi i z_1 - 2\pi i z_2 + v}\rt)$ with respect to $v$, we have
\begin{align}
&\Li\lt(e^{2\pi i \lt(\frac{M_2}{N+\frac{1}{2}}-1\rt) - 2\pi i z_1 - 2\pi i z_2 - \frac{\pi i}{N+\frac{1}{2}}}\rt) \notag\\
&= \Li\lt(e^{2\pi i \lt(\frac{M_2}{N+\frac{1}{2}}-1\rt) - 2\pi i z_1 - 2\pi i z_2}\rt)+\frac{\pi i}{N+\frac{1}{2}}\log\lt(1-e^{2\pi i \lt(\frac{M_2}{N+\frac{1}{2}}-1\rt) - 2\pi i z_1 - 2\pi i z_2}\rt)
+  O\lt(\frac{1}{(N+\frac{1}{2})^2}\rt)
\end{align}
The result follows by applying the same argument to all the dilogarithm terms.
\end{proof}

\subsection{The asymptotics expansion formula for $J_{N,N}(WL, e^{\frac{2\pi i}{N+\frac{1}{2}}})$}\label{WLNN}
Next, we consider the $(N,N)$-th colored Jones polynomials for the Whitehead Link. Note that in this case, (\ref{t1}) and (\ref{t5}) can be simplified as
\begin{align*}
 t^{\frac{N(2n+1)}{2}}-t^{-\frac{N(2n+1)}{2}} 
&= 2i\sin\lt( \frac{\pi(n+\frac{1}{2})}{N+\frac{1}{2}} \rt) \\
t^{-N(l+n)} 
&= e^{\pi i\lt( \frac{l}{N+\frac{1}{2}} + \frac{n}{N+\frac{1}{2}} \rt)}
\end{align*}
Define
\begin{align}
\tilde\Phi^{(1,1)}_{N,N} (z_1,z_2) 
&= \frac{1}{2\pi i} \lt[\varphi_r \lt( \frac{N\pi}{N+\frac{1}{2}}  - \pi z_1 - \pi z_2 -\frac{\pi}{2N+1}\rt) \rt. \notag\\
&\qquad  - \varphi_r\lt(  \frac{N\pi}{N+\frac{1}{2}} - \pi z_2- \frac{\pi}{2N+1}\rt) \notag \\
&\qquad  + \varphi_r \lt(\pi z_2 + \frac{\pi}{2N+1}\rt)  - \varphi_r\lt( \pi z_1 + \pi z_2 + \frac{\pi}{2N+1}\rt) \notag\\
&\lt.\qquad  + \varphi_r\lt( \pi z_1 + \frac{\pi}{2N+1}\rt) \rt] 
\end{align}
We can see that 
\begin{align*}
&\lim_{N\to\infty} \tilde\Phi^{(1,1)}_{N,N} (z_1,z_2) \\
&= \Phi^{\pm(1,1)}(z_1,z_2)\\
&=  \frac{1}{2\pi i} \lt[ \Li\lt(e^{- 2\pi i z_1 - 2\pi i z_2}\rt) - \Li\lt( e^{ - 2\pi i z_2 }\rt) + \Li \lt(e^{2\pi i z_2}\rt) - \Li\lt(e^{2\pi iz_1+ 2\pi i z_2}\rt)
 + \Li\lt(e^{2\pi i z_1}\rt)  \rt]
\end{align*}

Since $\Phi^{+(1,1)}(z_1,z_2)$ and $\Phi^{-(1,1)}(z_1,z_2)$ are the same, for simplicity we use $\Phi^{(1,1)}(z_1,z_2)$ to denote $\Phi^{\pm(1,1)}(z_1,z_2)$. Recall that for $\theta \in (0,\pi)$, we have 
$$ \Li (e^{2i\theta}) = \Li(1) + \theta(\pi - \theta) + 2i \Lambda(\theta)  ,$$
where $\Lambda$ is the Lobachevsky function defined by
$$ \Lambda(\theta) = -\int_0^\theta \log|2\sin t| dt $$
Also, the Lobachevsky function $\Lambda$ is an odd function. Let $$\Delta = \{(x,y) \in [0,1]^2 \mid x+y \leq 1\}$$
Then for $(z_1,z_2) \in \Delta$, the real part of $\Phi^{(1,1)}(x_1,x_2)$ is given by
\begin{align*}
\Re \Phi^{(1,1)} (x_1,x_2) 
&=  \frac{1}{\pi } \lt[ \Lambda\lt(- \pi  z_1 - \pi  z_2\rt) - \Lambda\lt(  - \pi  z_2 \rt) + \Lambda \lt(\pi  z_2\rt) - \Lambda\lt(\pi z_1+ \pi  z_2\rt) + \Lambda\lt(\pi z_1\rt) \rt] \\
&= \frac{1}{\pi} [-2 \Lambda(\pi x_1 + \pi x_2) + 2\Lambda(\pi x_2) + \Lambda(\pi x_1)]\\
&= \frac{1}{\pi} f(x_1,x_2)
\end{align*}

Note that this function $f(x_1, x_2)$ is the function appeared in Equation (3.22) of \cite{HZ07}. In particular, the real part of $\Phi^{(1,1)}(x_1,x_2)$ has a unique maximum at $(\frac{1}{2}, \frac{1}{4})$ in the $\Delta$, with maximum value 
$$ \Re\Phi^{(1,1)} \lt(\frac{1}{2}, \frac{1}{4}\rt) = \frac{1}{\pi}f \lt(\frac{1}{2}, \frac{1}{4}\rt) = \frac{1}{2\pi} \lt[8\Lambda\lt(\frac{\pi}{4}\rt)\rt] = \frac{1}{2\pi} v_8 = \frac{1}{2\pi} \Vol(WL)$$
Furthermore, the critical point equations of the potential function are given by
\begin{empheq}[left = \empheqlbrace]{align}
&\quad \log(1-e^{-2\pi i z_1 - 2\pi i z_2}) + \log(1 - e^{2\pi i z_1 + 2\pi i z_2}) - \log(1 - e^{2\pi i z_1}) = 0 \label{1meridian1}\\ \notag \\
&\quad \log(1-e^{-2\pi i z_1 - 2\pi i z_2}) + \log(1 - e^{2\pi i z_1 + 2\pi i z_2}) - \log(1 - e^{- 2\pi i z_2}) - \log(1 - e^{2\pi i z_2}) = 0 \label{1long1} 
\end{empheq}

Note that the point $(z_1, z_2) =  \lt(\frac{1}{2}, \frac{1}{4}\rt)$ indeed solves the equations $(\ref{1meridian1})$ and $(\ref{1long1})$. Moreover, at this critical point,
\begin{align*}
\frac{\partial^2}{\partial z_1^2} \Phi^{(1,1)} \lt(\frac{1}{2}, \frac{1}{4}\rt) &= (2\pi i) \lt( \frac{i}{1-i} - \frac{ -i}{1+i} + \frac{-1}{2} \rt) = 2\pi i \lt( \frac{2i-1}{2} \rt) \\ 
\frac{\partial^2}{\partial z_2^2} \Phi^{(1,1)} \lt(\frac{1}{2}, \frac{1}{4}\rt) &= (2\pi i) \lt( \frac{i}{1-i} - \frac{-i}{1+i} - \frac{-i}{1+i} + \frac{i}{1-i} \rt) = 2\pi i (2i)\\
\frac{\partial^2}{\partial z_1 \partial z_2} \Phi^{(1,1)}  \lt(\frac{1}{2}, \frac{1}{4}\rt) &= (2\pi i) \lt( \frac{i}{1-i} - \frac{ -i}{1+i} \rt) = 2\pi i (i)
\end{align*}

As a result, since
\begin{align*}
-\det (\text{Hess} \Phi^{(1,1)})  \lt(\frac{1}{2}, \frac{1}{4}\rt)  = -4\pi^2 (1+i) \neq 0,
\end{align*}
$\Phi^{(1,1)}(z_1,z_2)$ has a non-degenerate critical point at $\lt(\frac{1}{2},\frac{1}{4}\rt)$ with critical value
\begin{align}
&\Phi^{(1,1)}\lt(\frac{1}{2},\frac{1}{4}\rt)  \notag\\
&= \frac{1}{2\pi i} \lt[ 2\Li(i) - 2\Li(-i) + \Li(-1) \rt] \notag \\
&= \frac{1}{2\pi i} \lt[ 2\lt(\Li(1) + \lt(\frac{\pi}{4}\rt)\lt(\pi - \frac{\pi}{4}\rt) + 2i \Lambda\lt(\frac{\pi}{4}\rt) \rt) - 2 \lt(\Li(1) + \lt(\frac{3\pi}{4}\rt)\lt(\pi - \frac{3\pi}{4}\rt) + 2i \Lambda\lt(\frac{3\pi}{4}\rt) \rt) - \frac{\pi^2}{12} \rt] \notag\\
&= \frac{1}{2\pi} \lt[ 8\Lambda\lt(\frac{\pi}{4}\rt) + \frac{\pi^2 i}{12} \rt] \notag\\
&= \frac{1}{2\pi} \lt[ \Vol(\SS^3\backslash WL) + \frac{\pi^2 i}{12} \rt] 
\end{align}

We denote this non-degenerate critical point by $\mathbf{z}^{(1,1)}$. Similarly, we let $\mathbf{z}^{(1,1)}_{N,N}$ to be the critical point of the function $\tilde\Phi^{(1,1)}_{N,N}(z_1,z_2)$ with $\lim_{N\to \infty}\mathbf{z}^{(1,1)}_{N,N}= (\frac{1}{2},\frac{1}{4})$. 

Next, we are going to study the asymptotic expansion formula for $J_{N,N}(WL, e^{\frac{2\pi i}{N+\frac{1}{2}}})$ and show that the exponential growth is given by the critical value of $\Phi^{(1,1)} \lt(\mathbf{z}^{(1,1)}\rt)$. From previous discussion, since the critical value is the unique maximum of $\Re\Phi(z_1,z_2)$ on $\Delta$, it sufficies to study the asymptotic expansion formula near this critical point. 

Since $\mathbf{z}^{(1,1)}$ is a critical point, we can find a 4 real dimension ball $B_R \subset \CC^2 \cong \RR^4$ centred at the non-degenerate critical point with radius $R$ such that whenever $(x_1\pm\eta i,x_2), (x_1, x_2\pm\eta i) \in B_R $. we have  
\begin{align}\label{ball}
 \lt| \frac{d}{d \eta} \Re\Phi(x_1\pm\eta i,x_2) \rt|, \lt| \frac{d}{d\eta} \Re\Phi(x_1, x_2\pm\eta i)  \rt| < \pi 
\end{align}
In particular, we may choose $R$ small enough such that the intersection $\partial B_R \cap \Delta$ is a circle centered at the critical point. Let
\begin{align*}
S^\pm_1 &= \{ (x_1\pm\eta i, x_2) \mid (x_1,x_2) \in \partial B_R \cap \Delta, \eta \in [0,R /2]  \}
\cup \{ (x_1\pm ( R/2 ) i, x_2) \mid (x_1,x_2) \in B_R \cap \Delta \} \\
S^\pm_2 &= \{ (x_1, x_2\pm\eta i) \mid (x_1,x_2) \in \partial B_R \cap \Delta, \eta \in [0,R /2]  \} 
\cup \{ (x_1, x_2\pm ( R/2 ) i) \mid (x_1,x_2) \in B_R \cap \Delta \}
\end{align*}
For every points on $S_1^\pm$, we have
\begin{align}\label{S1}
\Re\Phi^{(1,1)}(x_1 \pm \eta i, x_2) - \Re\Phi^{(1,1)}(\mathbf{z}^{(1,1)}) \leq \lt[\Re\Phi^{(1,1)}(x_1,x_2) - \Re\Phi^{(1,1)}(\mathbf{z}^{(1,1)}) \rt] + \pi \eta < 2\pi \eta
\end{align}
Similarly, for every points on $S_2^\pm$, we have
\begin{align}\label{S2}
\Re\Phi^{(1,1)}(x_1, x_2 \pm \eta i) - \Re\Phi^{(1,1)}(\mathbf{z}^{(1,1)}) \leq \lt[\Re\Phi^{(1,1)}(x_1,x_2) - \Re\Phi^{(1,1)}(\mathbf{z}^{(1,1)}) \rt] + \pi \eta < 2\pi \eta
\end{align}
Thus, by the Poisson Summation Formula (Proposition 4.6, Remark 4.7 and 4.8 in \cite{O52}), we have
\begin{align}
&J_{N,N} (WL, e^{\frac{2\pi i}{N+\frac{1}{2}}}) \notag\\
&=  \frac{ e^{\frac{2\pi i}{N+\frac{1}{2}}(N^2 - \frac{N}{2} - \frac{1}{2})}}{\sin(\frac{\pi}{N+\frac{1}{2}})} \exp\lt(  -\varphi_r \lt(  \frac{\pi}{2N+1}\rt)\rt)\sum_{n=0}^{N-1}\sum_{l=1}^{N-1-n} \sin\lt( \pi \lt(\frac{n}{N+\frac{1}{2}}\rt) + \frac{\pi}{2N+1} \rt) e^{\pi i \lt(\frac{n}{N+\frac{1}{2}}+\frac{l}{N+\frac{1}{2}}\rt)} \notag\\
&\qquad \times
\exp\lt( \lt(N+\frac{1}{2}\rt) \tilde\Phi^{\pm(1,1)}_{N,N} \lt(\frac{n}{N+\frac{1}{2}},\frac{l}{N+\frac{1}{2}}\rt) \rt) \notag \\
&\stackrel[N \to \infty]{\sim}{} \frac{ e^{\frac{2\pi i}{N+\frac{1}{2}}(N^2 - \frac{N}{2} - \frac{1}{2})}}{\sin(\frac{\pi}{N+\frac{1}{2}})} \exp\lt(  -\varphi_r \lt(  \frac{\pi}{2N+1}\rt)\rt) \lt(N+\frac{1}{2}\rt)^2\notag\\
&\qquad \times \int\int_{B_R \cap \Delta} \sin\lt( \pi z_1 + \frac{\pi}{2N+1} \rt) e^{\pi i \lt(z_1+z_2\rt)} 
\exp\lt( \lt(N+\frac{1}{2}\rt) \tilde\Phi^{(1,1)}_{N,N} \lt(z_1,z_2\rt) \rt) dz_1 dz_2
\end{align}

Next, we want to apply the Saddle point approximation. A key observation is that we can choose $r$ as small as possible such that the assumption in Proposition 3.5 in \cite{O52} is automatically satisfied. Moreover, by Lemma A.3 of \cite{O52}, we have
\begin{align}\label{QDFcoe}
\exp\lt(-\varphi^h \lt( \frac{\pi}{2N+1} \rt) \rt) 
&= \exp\lt(- \frac{N+\frac{1}{2}}{2\pi i} \frac{\pi^2}{6} - \frac{1}{2}\log (N+\frac{1}{2}) - \frac{\pi i}{4} + \frac{\pi i}{12 (N+\frac{1}{2})} \rt) \notag\\
&\stackrel[N \to \infty]{\sim}{ } e^{-\frac{\pi i}{4}}(N+\frac{1}{2})^{-\frac{1}{2}}\exp\lt( \lt( N+\frac{1}{2}\rt) \frac{\pi i}{12} \rt)
\end{align}

Altogether, by the saddle point approximation (Proposition 3.5 in \cite{O52}), we have
\begin{align}
&\quad J_{N, N} (WL, e^{\frac{2\pi i}{N+\frac{1}{2}}}) \notag  \\
&\stackrel[N \to \infty]{\sim}{}  - e^{-\frac{\pi i}{4}} \frac{(N+\frac{1}{2})^{\frac{1}{2}}}{\pi}\exp\lt( \frac{N+\frac{1}{2}}{2\pi } \frac{\pi^2 i}{6} \rt) 
(N+\frac{1}{2})^2 \notag \\
&\qquad \int\int_{B_R \cap \Delta} \sin\lt( \pi z_1 + \frac{\pi}{2N+1} \rt) e^{\pi i \lt(z_1+z_2\rt)} 
\exp\lt( \lt(N+\frac{1}{2}\rt) \tilde\Phi^{(1,1)}_{N,N} \lt(z_1,z_2\rt) \rt) dz_1 dz_2 \notag \\
&\stackrel[N \to \infty]{\sim}{}  \frac{\pi(N+\frac{1}{2})^{3/2}}{\sqrt{ -4 \pi^2(1+i) }} 
\times\exp\lt( \lt(N+\frac{1}{2}\rt)\lt( \tilde\Phi^{(1,1)}_{N,N} (\mathbf{z}^{(1,1)}_{N,N}) + \frac{\pi i}{12} \rt)  \rt) \lt(1 + O\lt(\frac{1}{N}\rt) \rt)
\end{align}

with
\begin{align}
\lim_{N\to \infty}
\lt[ \tilde\Phi^{(1,1)}_{N,N} (\mathbf{z}^{(1,1)}_{N,N}) + \frac{\pi i}{12} \rt]
= \frac{1}{2\pi}\lt[ \Vol(\SS^3\backslash WL) + \frac{\pi^2 i }{4} \rt]
= \frac{1}{2\pi}[\Vol(\SS^3\backslash WL)+ i \CS(\SS^3\backslash WL)]
\end{align}

\begin{remark}
If we use the normalization in \cite{HZ07} and \cite{V08}, then there will be an extra factor
$$ \frac{t^{\frac{1}{2}}-t^{-\frac{1}{2}}}{t^{N/2}-t^{-N/2}} = \frac{2i\sin\lt(\pi (\frac{1}{N+\frac{1}{2}}) \rt)}{2i\sin\lt(\pi (\frac{N}{N+\frac{1}{2}}) \rt)} \stackrel[N \to \infty]{\sim}{} 2 $$
In particular, the exponent of $(N+\frac{1}{2})$ is $\frac{3}{2}$, which agrees with the result in \cite{HZ07} and \cite{V08}.
\end{remark}

\subsection{The asymptotics expansion formula for $J_{M_1,M_2}(WL,e^{\frac{2\pi i}{N+\frac{1}{2}}})$}\label{WLs}
In this section, we generalize the previous argument to study the asymptotics expansion formula for $J_{M_1,M_2}(WL,e^{\frac{2\pi i}{N+\frac{1}{2}}})$. Recall that the potential functions $\Phi^{\pm(s_1, s_2)} (z_1,z_2)$ are given by

\begin{align*}
\Phi^{\pm(s_1, s_2)} (z_1,z_2) 
&=  \frac{1}{2\pi i} \lt[ \pm \lt(2\pi i \lt( s_1 - 1\rt)\rt) \lt(2\pi i \lt(z_1 -\frac{1}{2} \rt)\rt)  - (2\pi i (s_2-1))(2\pi i z_1 + 2\pi i z_2)  \rt. \\
&\qquad + \Li\lt(e^{2\pi i (s_2-1) - 2\pi i z_1 - 2\pi i z_2}\rt) - \Li\lt( e^{2\pi i (s_2-1) - 2\pi i z_2 }\rt) \\
&\qquad \lt.+ \Li \lt(e^{2\pi i z_2}\rt) - \Li\lt(e^{2\pi iz_1+ 2\pi i z_2}\rt) + \Li\lt(e^{2\pi i z_1}\rt)  \rt]
\end{align*}

From now on, for simiplicity we consider the potential function $\Phi^{+(s_1,s_2)}(z_1,z_2)$. The same arguement works for the potential function $\Phi^{-(s_1,s_2)}(z_1,z_2)$. The following lemma gaurantees the existence of the critical points for the potential functions.
\begin{lemma} There exists a neighborhood $U\subset \CC^2$ of $(1,1) \in \CC^2$ and two holomorphic families of non-degenerate critical points $\{\mathbf{z}^{\pm(s_1,s_2)} \}_{(s_1,s_2) \in U}$ for the families of potential functions $\Phi^{\pm(s_1,s_2)}(z_1,z_2)$ such that $\mathbf{z}^{\pm(1,1)} = \lt(\frac{1}{2}, \frac{1}{4} \rt)$. 
\end{lemma}

\begin{proof} First of all, we pick a small enough neighborhood $U \subset \CC^2$ of the point $(1,1)\in \CC^2$ such that the function $F: U\times B_R \to \CC$ defined by
$$F(s_1,s_2, z_1,z_2) = \Phi^{+(s_1,s_2)}(z_1,z_2)$$
is well-defined and holomorphic. Consider the function $G:U \times B_R \to \CC^2$ given by
$$ G(s_1,s_2,z_1,z_2) = \lt(\frac{\partial}{\partial z_1} F(s_1,s_2, z_1,z_2), \frac{\partial}{\partial z_2} F(s_1,s_2, z_1,z_2)\rt) $$
By direct computation, we have
\begin{align}
\frac{\partial}{\partial z_1} F(s_1,s_2, z_1,z_2)
&= 2\pi i\lt[(s_1 - 1) - (s_2 - 1) \rt] + \log(1-e^{2\pi i (s_2 -1)}e^{-2\pi i z_1 - 2\pi i z_2})  \notag\\
& \qquad + \log(1 - e^{2\pi i z_1 + 2\pi i z_2}) - \log(1 - e^{2\pi i z_1}) \\
\frac{\partial}{\partial z_2} F(s_1,s_2, z_1,z_2) 
&= -  2\pi i  (s_2 - 1) + \log(1-e^{2\pi i (s_2 -1)}e^{-2\pi i z_1 - 2\pi i z_2}) \notag\\
&\qquad + \log(1 - e^{2\pi i z_1 + 2\pi i z_2}) - \log(1 - e^{2\pi i (s_2 - 1) - 2\pi i z_2}) \notag \\
&\qquad - \log(1 - e^{2\pi i z_2}) 
\end{align}
Moreover,
\begin{align}
\frac{\partial}{\partial s_1} \lt( \frac{\partial}{\partial z_1} F(s_1,s_2, z_1,z_2) \rt)
&= 2\pi i  \\
\frac{\partial}{\partial s_2} \lt( \frac{\partial}{\partial z_1} F(s_1,s_2, z_1,z_2) \rt)
&= 0 \\
\frac{\partial}{\partial s_2} \lt( \frac{\partial}{\partial z_2} F(s_1,s_2, z_1,z_2) \rt)
&= -2\pi i - \frac{2\pi i e^{2\pi i (s_2 -1)}e^{-2\pi i z_1 - 2\pi i z_2}}{1-e^{2\pi i (s_2 -1)}e^{-2\pi i z_1 - 2\pi i z_2}} + \frac{2\pi i e^{2\pi i (s_2 -1)}e^{- 2\pi i z_2}}{1-e^{2\pi i (s_2 -1)}e^{- 2\pi i z_2}} 
\end{align}
As a result, we have
\begin{align}
\begin{vmatrix}
\frac{\partial}{\partial s_1} \lt( \frac{\partial}{\partial z_1} F \rt)
&
\frac{\partial}{\partial s_2} \lt( \frac{\partial}{\partial z_1} F \rt) \\
\frac{\partial}{\partial s_1} \lt( \frac{\partial}{\partial z_2} F \rt) 
&
\frac{\partial}{\partial s_2} \lt( \frac{\partial}{\partial z_2} F\rt)
\end{vmatrix}
\lt(1,1,\frac{1}{2}, \frac{1}{4}\rt)
= -2\pi i (1+i) 
\neq 0
\end{align}
The result follows from the implicit function theorem.
\end{proof}

Note that on the surfaces $S_1^+$, from (\ref{S1}) we have
\begin{align}
\Re\Phi^{(1,1)}(x_1 \pm \eta i, x_2) - \Re\Phi^{(1,1)}(\mathbf{z}^{(1,1)}) 
< 2\pi \eta
\end{align}
Since $S_1^+$ is a compact set, by continuity, we may choose $\delta$ small enough such that 
\begin{align*}
\Re\Phi^{(s_1,s_2)}(x_1 + \eta i, x_2) - \Re\Phi^{(s_1,s_2)}(\mathbf{z}^{+(s_1,s_2)}) 
< 2\pi \eta
\end{align*}
for any $1-\delta<s_1,s_2 \leq 1$. The same argument works for $S_1^-, S_2^+$ and $S_2^-$.

As a result, the assumptions for applying the Poisson Summation formula are satisfied. Finally, for the same reason, we can choose $\delta>0$ small enough such that the conditions for applying the saddle point approximation are automatically satisfied. Together with Lemma~\ref{EM_2}, the asymptotic expansion formula of $J_{M_1,M_2}(WL,e^{\frac{2\pi i}{N+\frac{1}{2}}})$ is given by

\begin{align}
&\quad J_{M_1, M_2} (WL, e^{\frac{2\pi i}{N+\frac{1}{2}}}) \notag  \\
&= \frac{ (-1)^{(M_1 - N) }  e^{\frac{2\pi i}{N+\frac{1}{2}}(M_2^2 - \frac{M_2}{2} - \frac{1}{2})}}{2\sin(\frac{\pi}{N+\frac{1}{2}})} \exp\lt(  -\varphi_r\lt(  \frac{\pi}{2(N+\frac{1}{2})}\rt)\rt)\notag \\
&\qquad  \lt[e^{\pi i \lt(\frac{M_1}{N+\frac{1}{2}} - 1\rt) }\sum_{n=0}^{M_2-1} \sum_{l=0}^{M_2-1-n} \exp\lt( \lt(N+\frac{1}{2}\rt)\Phi^{+(s_1,s_2)}_{M_1,M_2} \lt(\frac{n}{N+\frac{1}{2}},\frac{l}{N+\frac{1}{2}}\rt) \rt) \rt. \notag \\
&\qquad \quad \lt. + e^{-\pi i \lt(\frac{M_1}{N+\frac{1}{2}} - 1\rt) }\sum_{n=0}^{M_2-1} \sum_{l=0}^{M_2-1-n} \exp\lt( \lt(N+\frac{1}{2}\rt)\Phi^{-(s_1,s_2)}_{M_1,M_2} \lt(\frac{n}{N+\frac{1}{2}},\frac{l}{N+\frac{1}{2}}\rt) \rt) \rt] \notag \\
&\stackrel[N \to \infty]{\sim}{}  \frac{ e^{-\frac{\pi i}{4}} (-1)^{(M_1 - N) }  e^{\frac{2\pi i}{N+\frac{1}{2}}(M_2^2 - \frac{M_2}{2} - \frac{1}{2})}}{2\pi} \lt(N+\frac{1}{2}\rt)^{\frac{1}{2}}\exp\lt( \frac{N+\frac{1}{2}}{2\pi } \frac{\pi^2 i}{6} \rt)\lt(N+\frac{1}{2}\rt)^2   \notag  \\
& \qquad\qquad\lt[ e^{\pi i (\frac{M_1}{N+\frac{1}{2}} -1)} \int\int_{B_R \cap \Delta}  E_{M_2}(z_1,z_2)\exp \lt( \lt(N+\frac{1}{2}\rt) \Phi^{+\lt(\frac{M_1}{N+\frac{1}{2}}, \frac{M_2}{N+\frac{1}{2}}\rt)} (z_1, z_2) \rt) dz_1 dz_2 \rt. \notag  \\
& \qquad\qquad\qquad \lt.  + e^{-\pi i (\frac{M_1}{N+\frac{1}{2}} -1)} \int\int_{B_R \cap \Delta}  E_{M_2}(z_1,z_2)\exp \lt( \lt(N+\frac{1}{2}\rt) \Phi^{-\lt(\frac{M_1}{N+\frac{1}{2}}, \frac{M_2}{N+\frac{1}{2}}\rt)} (z_1, z_2) \rt) dz_1 dz_2 \rt] \notag  \\
&\stackrel[N \to \infty]{\sim}{}  \frac{e^{-\frac{\pi i}{4}} (-1)^{(M_1 - N) }  e^{\frac{2\pi i}{N+\frac{1}{2}}(M_2^2 - \frac{M_2}{2} - \frac{1}{2})}}{2} \exp\lt( \frac{N+\frac{1}{2}}{2\pi } \frac{\pi^2 i}{6} \rt) \lt(N+\frac{1}{2}\rt)^\frac{3}{2}   \notag  \\
& \qquad\quad\lt[ e^{\pi i \lt(\frac{M_1}{N+\frac{1}{2}} -1\rt)}E_{M_2}(\mathbf{z}^{+(s_1,s_2)}_{M_1,M_2} ) \frac{ \exp\lt( \lt(N+\frac{1}{2}\rt) \Phi^{+\lt(\frac{M_1}{N+\frac{1}{2}}, \frac{M_2}{N+\frac{1}{2}}\rt)} \lt(\mathbf{z}^{+\lt(\frac{M_1}{N+\frac{1}{2}}, \frac{M_2}{N+\frac{1}{2}}\rt)} \rt) \rt)}{ \sqrt{-\det\Hess \Phi^{+\lt(\frac{M_1}{N+\frac{1}{2}}, \frac{M_2}{N+\frac{1}{2}}\rt)} \lt(\mathbf{z}^{+\lt(\frac{M_1}{N+\frac{1}{2}}, \frac{M_2}{N+\frac{1}{2}}\rt)}\rt)}}  \rt. \notag\\
&\qquad\qquad\qquad \lt.  + e^{-\pi i \lt(\frac{M_1}{N+\frac{1}{2}} -1\rt)}E_{M_2}(\mathbf{z}^{-(s_1,s_2)}_{M_1,M_2} )  \frac{ \exp\lt( \lt(N+\frac{1}{2}\rt) \Phi^{-\lt(\frac{M_1}{N+\frac{1}{2}}, \frac{M_2}{N+\frac{1}{2}}\rt)} \lt(\mathbf{z}^{-\lt(\frac{M_1}{N+\frac{1}{2}}, \frac{M_2}{N+\frac{1}{2}}\rt)} \rt) \rt)}{ \sqrt{-\det\Hess \Phi^{-\lt(\frac{M_1}{N+\frac{1}{2}}, \frac{M_2}{N+\frac{1}{2}}\rt)} \lt(\mathbf{z}^{-\lt(\frac{M_1}{N+\frac{1}{2}}, \frac{M_2}{N+\frac{1}{2}}\rt)} \rt)}}  \rt] ,
\end{align}
where the last line holds because those terms do not cancel out with each other. 
In fact, at the end of Section~\ref{diffsame}, we will show that for $s_1,s_2$ sufficiently close to $1$,
$$\Phi^{+(s_1,s_2)}(\mathbf{z}^{+(z_1,z_2)}) = 
\Phi^{-(s_1,s_2)}(\mathbf{z}^{-(z_1,z_2)}) $$
and
$$\Re\Phi^{+(s_1,s_2)}(\mathbf{z}^{+(z_1,z_2)})
= \Re\Phi^{-(s_1,s_2)}(\mathbf{z}^{-(z_1,z_2)})
= \Vol\lt(\SS^3 \backslash WL, u_1 = 2\pi i (1 - s_1), v_2=4\pi i (1 - s_2)\rt),$$
where $\Vol\lt(\SS^3 \backslash WL, u_1 = 2\pi i (1 - s_1), v_2=4\pi i (1 - s_2)\rt)  $ is the hyperbolic volume of the Whitehead link complement equipped with the incomplete hyperbolic structure such that the logarithm of the holonomy of the meridian of the belt and the logarithm of the holonomy of the longitude of the clasp are $u_1 = 2\pi i (1 - s_1)$ and $v_2=4\pi i (1 - s_2)$ respectively. As a result, by choosing $\delta>0$ sufficiently small, for any $1-\delta < s_1,s_2 \leq 1$, we have

\begin{align}
 J_{M_1, M_2} (WL, e^{\frac{2\pi i}{N+\frac{1}{2}}})
&\stackrel[N \to \infty]{\sim}{}  \frac{e^{-\frac{\pi i}{4}} (-1)^{(M_1 - N) }  e^{\frac{2\pi i}{N+\frac{1}{2}}(M_2^2 - \frac{M_2}{2} - \frac{1}{2})}}{2} \exp\lt( \frac{N+\frac{1}{2}}{2\pi } \frac{\pi^2 i}{6} \rt) \lt(N+\frac{1}{2}\rt)^\frac{3}{2}   \notag  \\
& \qquad\quad \lt[ \frac{e^{\pi i \lt(\frac{M_1}{N+\frac{1}{2}} -1\rt)}E_{M_2}(\mathbf{z}^{+(s_1,s_2)}_{M_1,M_2} )}{ \sqrt{-\det\Hess \Phi^{+\lt(\frac{M_1}{N+\frac{1}{2}}, \frac{M_2}{N+\frac{1}{2}}\rt)} \lt(\mathbf{z}^{+\lt(\frac{M_1}{N+\frac{1}{2}}, \frac{M_2}{N+\frac{1}{2}}\rt)}\rt)}}  \rt.\notag\\
& \lt.\qquad\quad +
\frac{e^{-\pi i \lt(\frac{M_1}{N+\frac{1}{2}} -1\rt)}E_{M_2}(\mathbf{z}^{-(s_1,s_2)}_{M_1,M_2} )}{ \sqrt{-\det\Hess \Phi^{-\lt(\frac{M_1}{N+\frac{1}{2}}, \frac{M_2}{N+\frac{1}{2}}\rt)} \lt(\mathbf{z}^{-\lt(\frac{M_1}{N+\frac{1}{2}}, \frac{M_2}{N+\frac{1}{2}}\rt)}\rt)}} \rt]\notag\\
&\qquad \qquad \exp\lt( \lt(N+\frac{1}{2}\rt) \Phi^{+\lt(\frac{M_1}{N+\frac{1}{2}}, \frac{M_2}{N+\frac{1}{2}}\rt)} \lt(\mathbf{z}^{+\lt(\frac{M_1}{N+\frac{1}{2}}, \frac{M_2}{N+\frac{1}{2}}\rt)} \rt) \rt)
\end{align}

\subsection{Geometry of the potential function}\label{geoWL}

To understand the critical values of the potential functions, we need to find a concrete triangulation of the link complement associated to the potential function. 
The triangulation is the one appeared in \cite{T79}. First of all, we prepare an ideal octahedra (Figure~\ref{f1}). The idea is to glue the two blue faces together and the two red faces together (Figure \ref{f2}) to form a cylinder. \\

          \begin{figure}[H]
          \centering
              \includegraphics[width=0.7\linewidth]{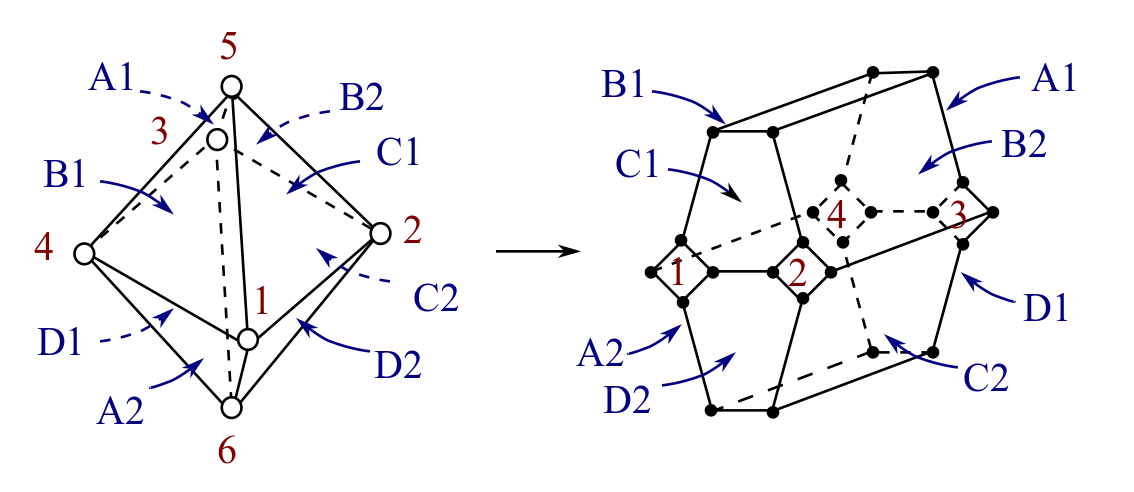}
              \caption{ideal octahedron}\label{f1}
          \end{figure}
          
          \begin{figure}[H]
          \centering
              \includegraphics[width=0.4\linewidth]{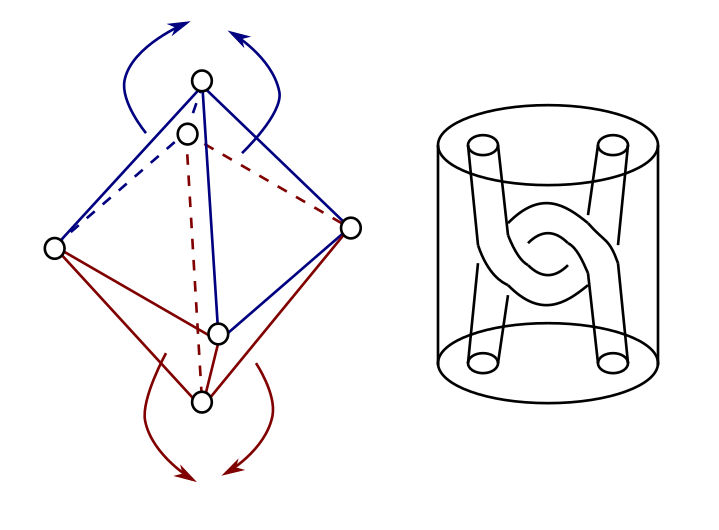}
              \caption{glue the faces with the same color together}\label{f2}
          \end{figure}

The following pictures illustrate how to obtain a cylinder from an ideal octahedron.\\
\begin{minipage}{\linewidth}
      \centering
      \begin{minipage}{0.45\linewidth}
          \begin{figure}[H]
              \includegraphics[width=\linewidth]{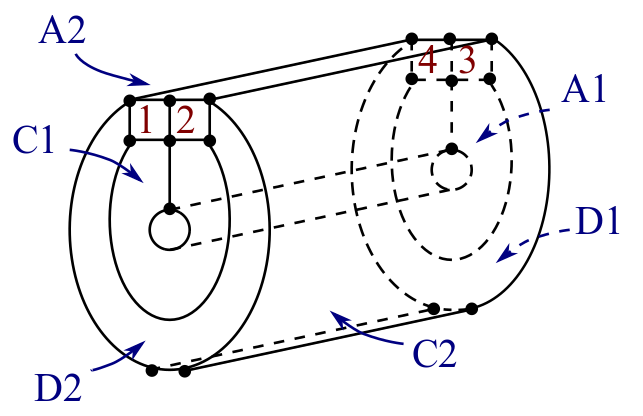}
              \caption{the faces $B_1$ and $B_2$ are glued together}
          \end{figure}
      \end{minipage}
      \hspace{0.05\linewidth}
      \begin{minipage}{0.45\linewidth}
          \begin{figure}[H]
              \includegraphics[width=\linewidth]{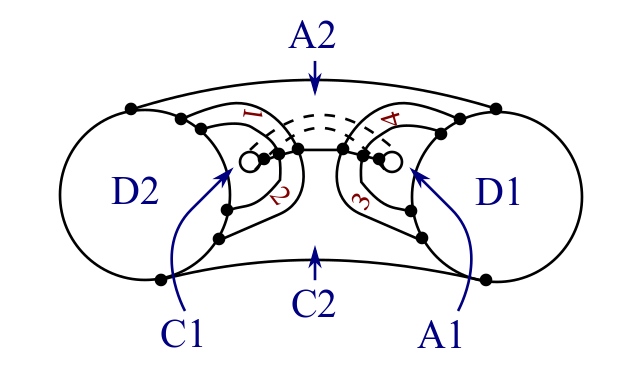}
              \caption{deform the object so that it looks like a cylinder}
          \end{figure}
      \end{minipage}
\end{minipage}
\begin{minipage}{\linewidth}
      \centering
      \begin{minipage}{0.45\linewidth}
          \begin{figure}[H]
              \includegraphics[width=\linewidth]{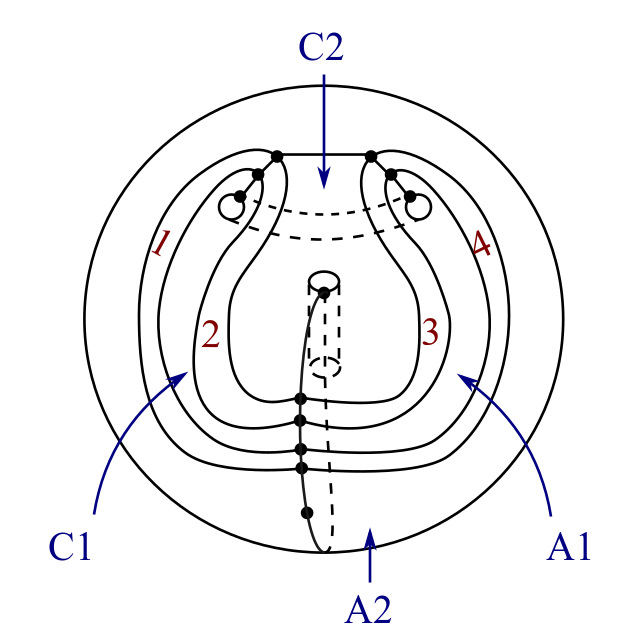}
              \caption{the faces $D_1$ and $D_2$ are glued together}
          \end{figure}
      \end{minipage}
            \hspace{0.05\linewidth}
            \begin{minipage}{0.4\linewidth}
          \begin{figure}[H]
              \includegraphics[width=\linewidth]{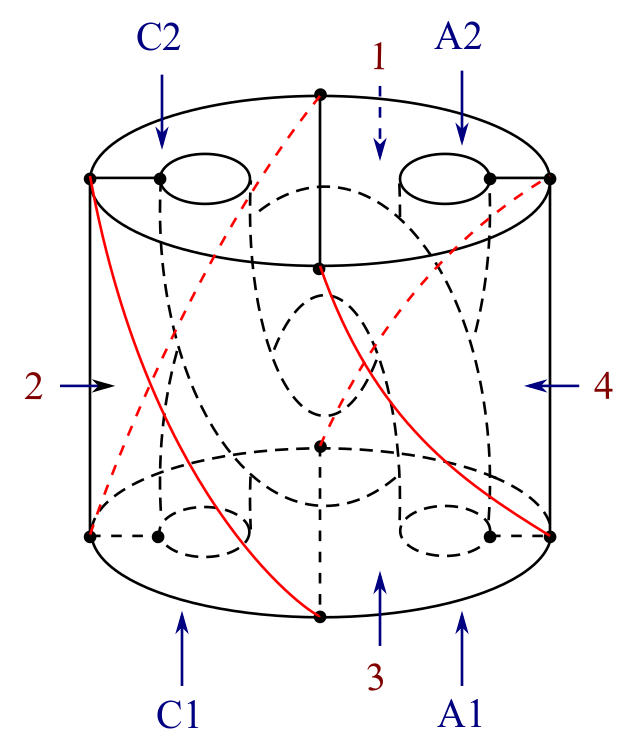}
              \caption{deform the object to obtain a cylinder with a clasp removed}
          \end{figure}
      \end{minipage}
  \end{minipage}\\

Next, we consider the curl face of the cylinder. The labels $1,2,3,4$ refer to the truncated faces of the ideal octahedron. If we glue the top of the cylinder to the bottom according to the rule $C_2 \to C_1$ and $A_2 \to A_1$, we will obtain a decomposition of the boundary torus into parallelograms. Note that this torus is exactly the boundary torus of a tubular neighborhood of the belt component of the Whitehead link.\\ 
          \begin{figure}[H]
          \centering
              \includegraphics[width=\linewidth]{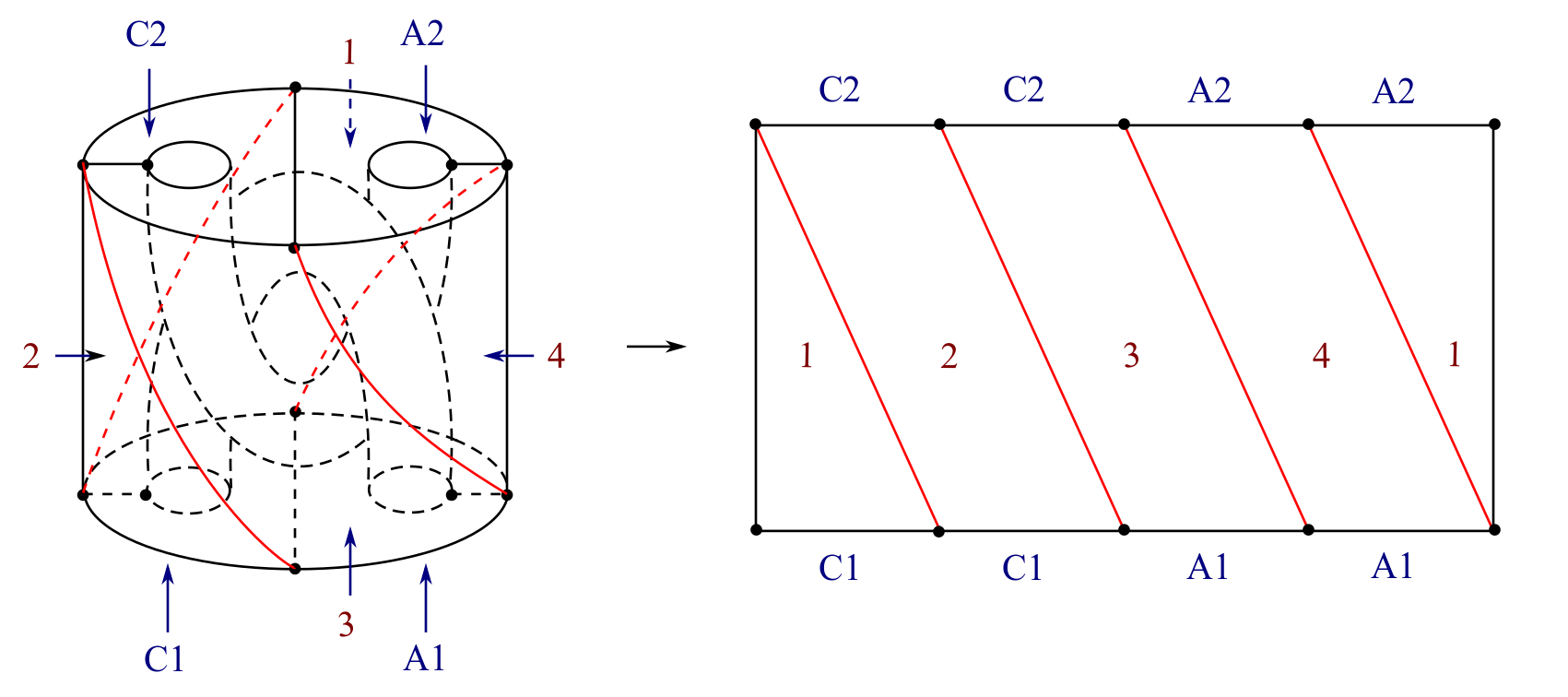}
              \caption{decomposition of the curl face into parallelograms}
          \end{figure}
 
Next, we decompose the ideal octahedron into 5 ideal tetrahedra and assign shape parameters to them (Figure \ref{f8}).  The parameters $U,Z$ and $W$ are related by $U=(ZW)^{-1}$. The decomposition of the truncated faces and the assignments of shape parameters to them are shown in Figure \ref{f9}.\\

          \begin{figure}[H]
                \centering
              \includegraphics[width=0.9\linewidth]{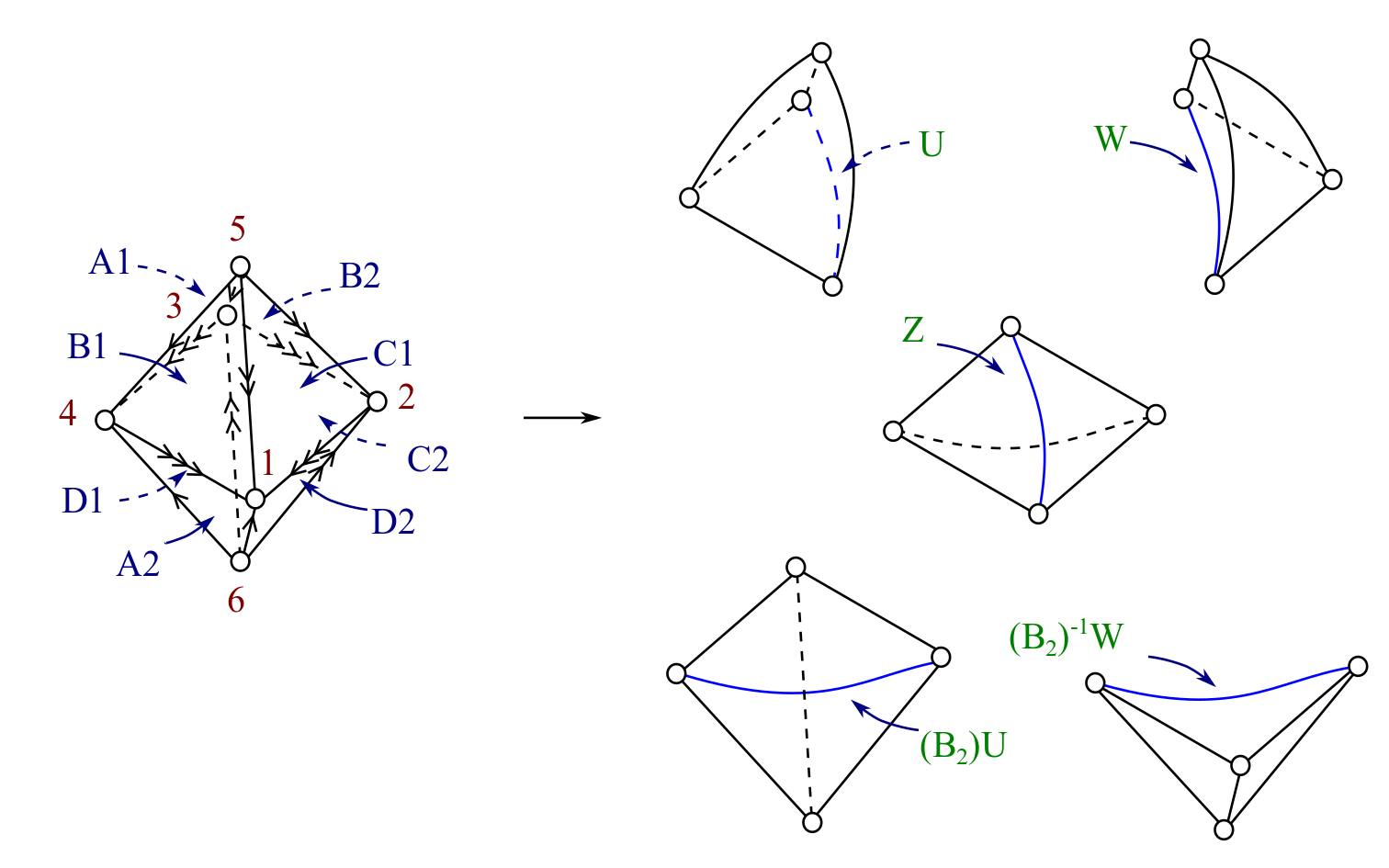}
              \caption{decomposition of the ideal octahedron into 5 ideal tetrahedra}\label{f8}
          \end{figure}
        
          \begin{figure}[H]
                \centering
              \includegraphics[width=\linewidth]{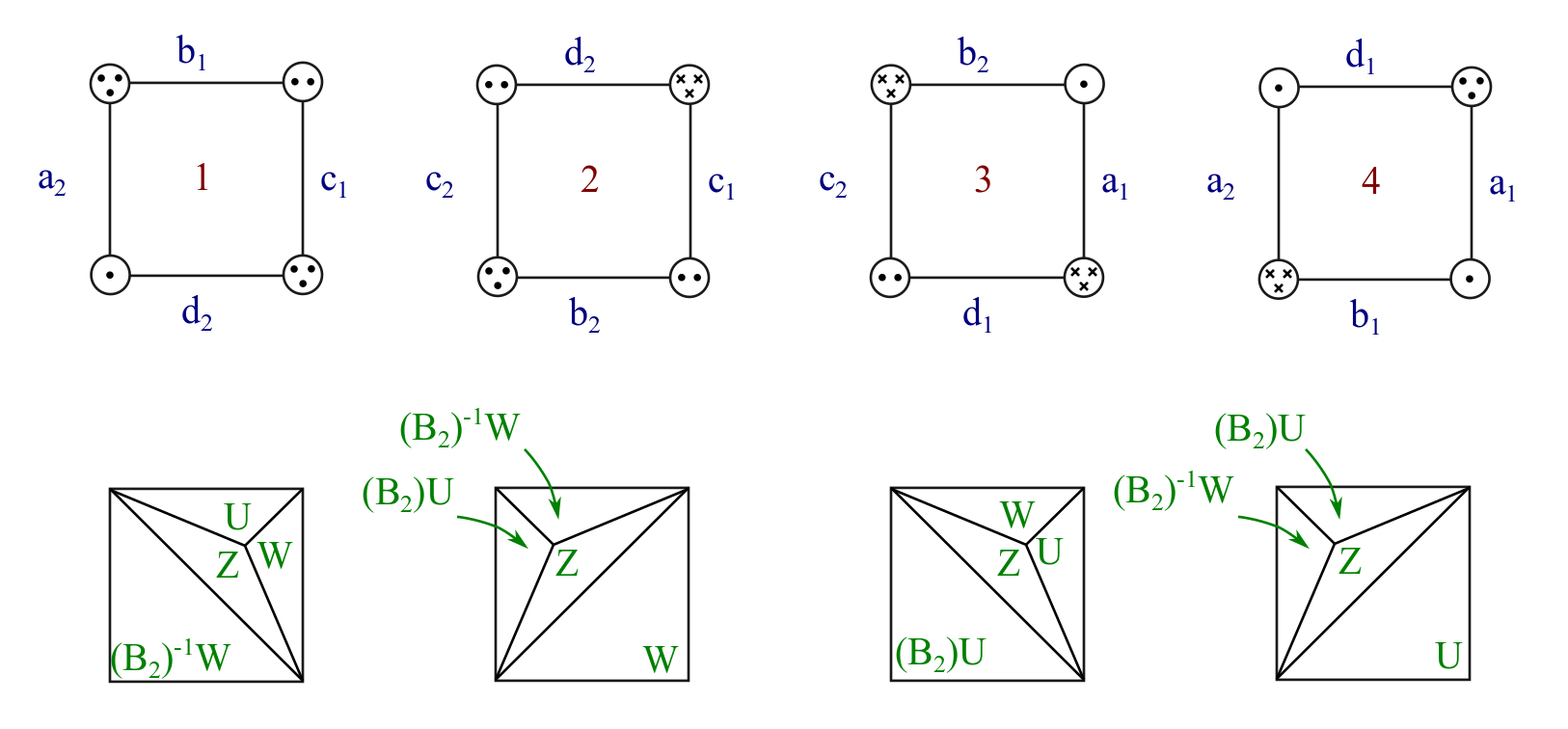}\label{f9}
              \caption{assignment of shape parameters}
              \label{f9}
          \end{figure}

After that, we slide the left edge down (Figure \ref{f10}) to obtain the geometric triangulation of the curl surface.\\

          \begin{figure}[H]
          \centering
              \includegraphics[width=0.8\linewidth]{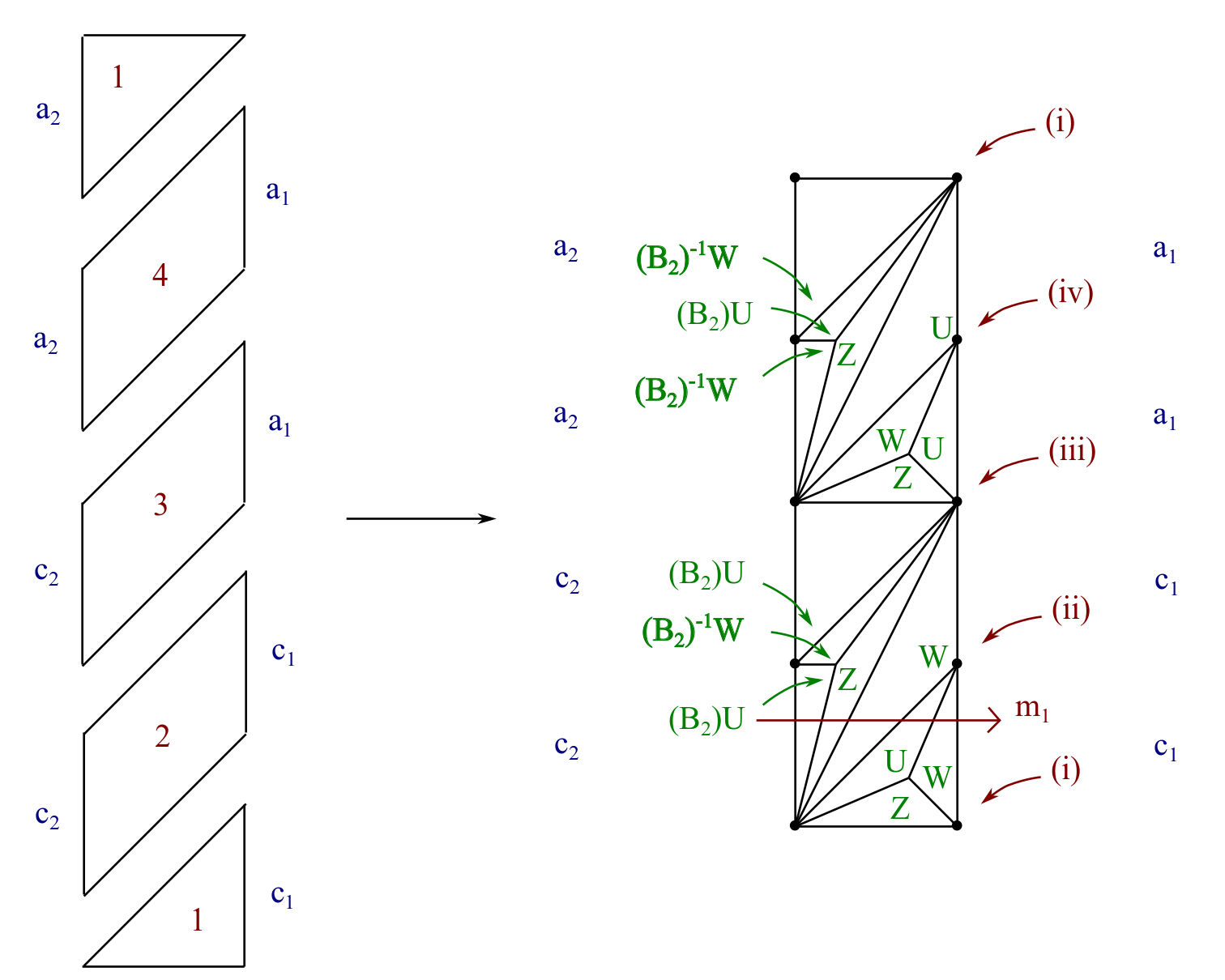}
              \caption{triangulation of the boundary torus of the belt component}\label{f10}
          \end{figure}

Recall that for each triangle, the shape parameters are given in counter-clockwise by $A$, $A' = \frac{1}{1-A}$ and $A'' = 1 - \frac{1}{A}$ respectively. From direct computation, for the boundary torus of the belt component, the edges equations are given by
\begin{itemize}
\item vertex (i): 
\begin{align}
[W'Z''(B_2^{-1}W)'(B_2U)'Z''U'][(B_2^{-1}W)''Z'U'W''Z'(B_2U)'']
&= \frac{-1}{W}\frac{-1}{Z}\frac{-1}{Z}\frac{-1}{B_2^{-1}W}\frac{-1}{B_2U}\frac{-1}{U}=1 
\end{align}
\item vertex (ii): 
\begin{align}
(WU'W'') ( (B_2U)' (B_2^{-1}W)'' (B_2U) )
= \lt(\frac{U'}{W'} \rt) \lt( \frac{(B_2U)''}{(B_2^{-1}W)''} \rt)^{-1}
= \frac{B_2\lt(1-\dfrac{B_2}{W}\rt)(1- W)}{(1-B_2U)\lt(1-\dfrac{1}{U}\rt)} \label{beltv2}
\end{align}
\item vertex (iii): 
\begin{align}
[W'Z''(B_2^{-1}W)'(B_2U)'Z''U'][(B_2^{-1}W)''Z'U'W''Z'(B_2U)'']
&= \frac{-1}{W}\frac{-1}{Z}\frac{-1}{Z}\frac{-1}{B_2^{-1}W}\frac{-1}{B_2U}\frac{-1}{U}=1
\end{align}
\item vertex (iv): 
\begin{align}
(UW'U'') ( (B_2^{-1}W)' (B_2U)'' (B_2^{-1}W) )
= \lt[\lt(\frac{U'}{W'} \rt) \lt( \frac{(B_2U)''}{(B_2^{-1}W)''} \rt)^{-1}\rt]^{-1}\
= \frac{(1-B_2U)\lt(1-\dfrac{1}{U}\rt)}{B_2\lt(1-\dfrac{B_2}{W}\rt)(1- W)} \label{beltv4}
\end{align}
\end{itemize}

For the top and bottom truncated faces, we have\\

          \begin{figure}[H]
          \centering
              \includegraphics[width=0.8\linewidth]{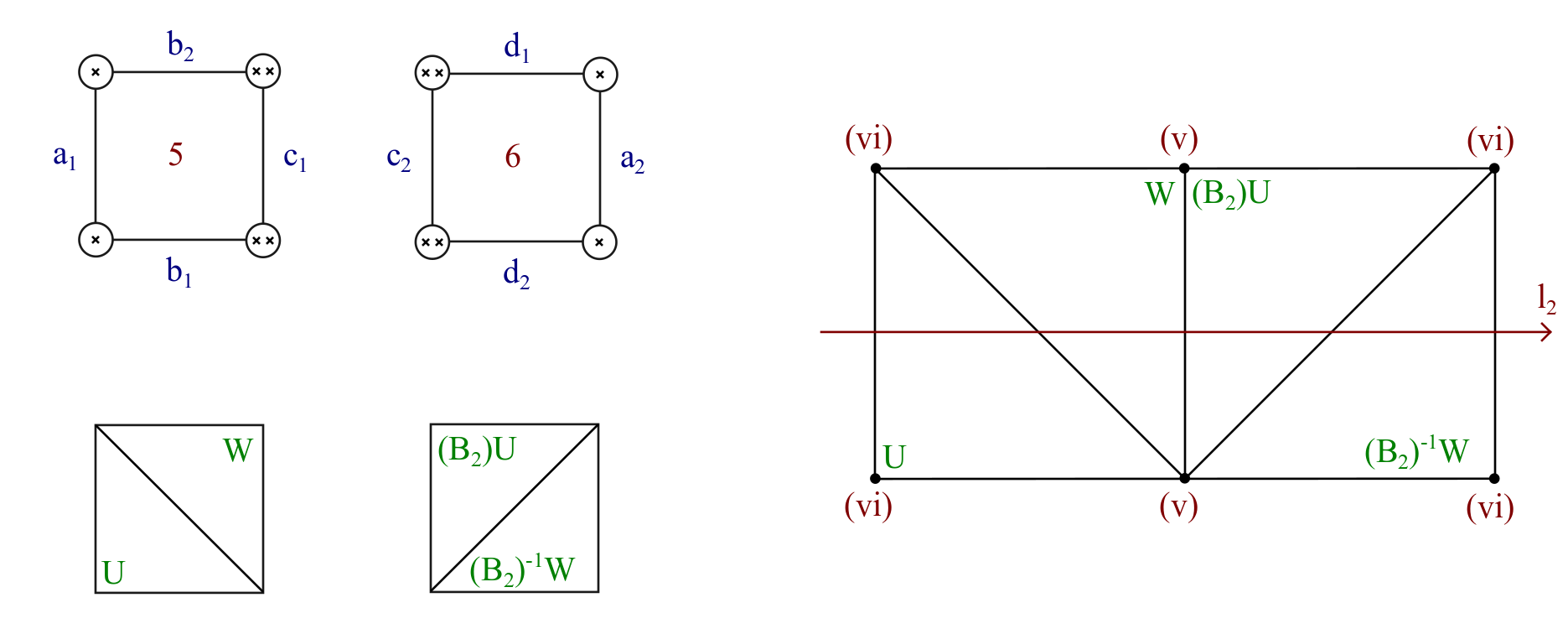}
              \caption{triangulation of the boundary torus of the belt component}\label{f11}
          \end{figure}
      
From direct computation, for the boundary torus of the clasp component, the edges equations are given by
\begin{itemize}
\item vertex (v): 
\begin{align}
(W) (B_2U)(B_2^{-1}W)''  (B_2U)' (W)'' (U') 
= \lt(\frac{U'}{W'} \rt) \lt( \frac{(B_2U)''}{(B_2^{-1}W)''} \rt)^{-1}
= \frac{B_2\lt(1-\dfrac{B_2}{W}\rt)(1- W)}{(1-B_2U)\lt(1-\dfrac{1}{U}\rt)}  \label{claspv1}
\end{align}
\item vertex (vi): 
\begin{align}
(B_2^{-1}W) (B_2U)''(W')  (U)'' (B_2^{-1}W)' (U) 
= \lt[\lt(\frac{U'}{W'} \rt) \lt( \frac{(B_2U)''}{(B_2^{-1}W)''} \rt)^{-1}\rt]^{-1}
= \lt[ \frac{B_2\lt(1-\dfrac{B_2}{W}\rt)(1- W)}{(1-B_2U)\lt(1-\dfrac{1}{U}\rt)} \rt]^{-1}\label{claspv2}
\end{align}
\end{itemize}

Now we can explore the correspondence between the critical point equations of the potential function and the edges equations of this triangulation. For $\Phi^{+(s_1,s_2)}(z_1,z_2)$, recall that the critical point equations are given by
\begin{empheq}[left = \empheqlbrace]{align}
&\quad \log(1-e^{2\pi i (s_2 -1)}e^{-2\pi i z_1 - 2\pi i z_2}) + \log(1 - e^{2\pi i z_1 + 2\pi i z_2}) - \log(1 - e^{2\pi i z_1}) \notag \\
&=  2\pi i\lt[-(s_1 - 1) + (s_2 - 1) \rt] \label{smeridian1}\\ \notag \\
&\quad \log(1-e^{2\pi i (s_2 -1)}e^{-2\pi i z_1 - 2\pi i z_2}) + \log(1 - e^{2\pi i z_1 + 2\pi i z_2}) - \log(1 - e^{2\pi i (s_2 - 1) - 2\pi i z_2})\notag\\
&\quad  - \log(1 - e^{2\pi i z_2}) \label{slong1}  \notag \\
&=  2\pi i  (s_2 - 1)
\end{empheq}

Put $Z=e^{2\pi i z_1}$, $W=e^{2\pi i z_2}$, $U= (ZW)^{-1}$, $B_1 = e^{2\pi i s_1}$ and $B_2 = e^{2\pi i s_2}$. After taking exponential, we have
\begin{empheq}[left = \empheqlbrace]{align}
\frac{\lt(1 - B_2 U \rt)\lt( 1 - \dfrac{1}{U}\rt)}{\lt( 1-Z \rt)} &= B_1^{-1}  B_2 \label{smeridian2} \\
\frac{(1-B_2U)\lt(1-\dfrac{1}{U}\rt)}{(1-\dfrac{B_2}{W})(1- W)} &= B_2 \label{slong2}
\end{empheq}

In particular, from $(\ref{slong2})$ we have
$$ \lt(\frac{U'}{W'} \rt) \lt( \frac{(B_2U)''}{(B_2^{-1}W)''} \rt)^{-1} = 1$$
i.e. the edge equations of the triangulation~(\ref{beltv2}), (\ref{beltv4}), (\ref{claspv1}) and (\ref{claspv2}) are satisfied. Furthermore, by (\ref{smeridian2}), the holonomy of the meridian of the belt component (Figure \ref{f10}) is given by
\begin{align}
m_1 = \frac{W''  U'}{W'' Z' (B_2 U)''} =\frac{B_2\lt( 1-Z \rt)}{\lt(1 - B_2 U \rt)\lt( 1 - \dfrac{1}{U}\rt)} = B_1 
\end{align}

By (\ref{slong2}), the holonomy of the longitude of the clasp component (Figure \ref{f11}) is given by
\begin{align}\label{slong2B2}
l_2 = \quad \frac{(B_2^{-1}W)'}{(B_2U)'} \cdot \frac{U''}{W''} =  \frac{(1-B_2U)\lt(1-\dfrac{1}{U}\rt)}{\lt(1-\dfrac{B_2}{W}\rt)(1- W)} \times B_2 = B_2^2
\end{align}

Finally, for the potential function $\Phi^{-(s_1,s_2)}(z_1,z_2)$, by replacing $B_1$ by $B_1^{-1}$ in the above, we have the same correspondence between critical point equations of the potential function and the hyperbolic gluing equation for the triangulation of the link complement.

%%%%%%%%%%%
\iffalse
The meridian and longitude of the clasp are related by the following figure:
          \begin{figure}[H]
          \centering
              \includegraphics[width=0.7\linewidth]{f12}
              \caption{$\Gamma=m^2_2$}\label{Gamma}
          \end{figure}
As a result, we have $m_2 = \sqrt{\Gamma}=B_2$ for suitable orientation of $m_2$. Similar arguments can be applied to $\Phi^{-(s_1,s_2)}(z_1,z_2)$ by replacing $B_1$ by $B_1^{-1}$. Next, recall that the longitude of a knot is defined to be a parallel copy of the knot with zero linking number. From Figure~\ref{Gamma}, since $\Gamma$ has linking number $2$, the longitude of the clasp component is given by
\begin{align}\label{NZL2}
l_2 = \Gamma \cdot m_2^{-2} = B_2^{-4}
\end{align}

Finally, to end this subsection, we derive two equations about the shape parameters which will be used later. First, since $m_2=B_2$, we have 
$$ \frac{U'}{W'}=B_2 $$ 
From (\ref{slong2B2}), we also have 
\begin{alignat}{2}\label{NZL1}
&& \frac{(1-\dfrac{B_2}{W})(1- W)}{(1-B_2U)\lt(1-\dfrac{1}{U}\rt)} \cdot \frac{1}{B_2} &= B_2^{-2}  \notag \\
\implies && \frac{U(1-\dfrac{B_2}{W})}{(1-B_2U)} \cdot \frac{U'}{W'} \cdot\frac{1}{B_2} &= - B_2^{-2} \notag \\
\implies && \frac{U\lt(1-\dfrac{B_2}{W}\rt)}{(1-B_2U)} 
&= - B_2^{-2} 
\end{alignat}
\fi 
%%%%%%%%%%%

\subsection{Differential formula for the potential function}\label{diffsame}
After we associate a triangulation of the link complement to the potential functions, we need to prove that the critical values of the potential functions $\Phi^{\pm(s_1,s_2)}(\mathbf{z}^{\pm(z_1,z_2)})$ are indeed the sum of the hyperbolic volume of the tetrahedra. This can be done by using the following differential formula satisfied by the potential functions.
\begin{lemma}
Write $z_k = x_k + i y_k$ for $k=1,2$. The real part of the potential function satisfies the following differential equation:
\begin{align} 
\Re \Phi^{\pm(s_1,s_2)}(z_1,z_2) = \frac{1}{2\pi} V^{\pm(s_1,s_2)}(z_1,z_2) + \sum_{k=1}^2 y_k \frac{\partial}{\partial y_k} \Re \Phi^{\pm(s_1,s_2)}(z_1,z_2),
\end{align}
where 
\begin{align}
V^{\pm(s_1,s_2)}(z_1,z_2) 
&= D(e^{2\pi i(s_2-1) - 2\pi i z_1 - 2\pi i z_2}) - D(e^{2\pi i (s_2-1) - 2\pi i z_2}) + D(e^{2\pi i z_2}) \notag \\
&\qquad - D(e^{2\pi i z_1 + 2\pi i z_2}) + D(e^{2\pi i z_1}) 
\end{align}
and $D(z)$ is the Bloch-Wigner function
\begin{align}
D(z) = \im\Li(z) + \log|z|\Arg(1-z)
\end{align}
\end{lemma}

\begin{proof}
First of all, for $z = x + i y$, since
\begin{align*}
\frac{d}{dz} \frac{\Li (e^{2\pi i z})}{2\pi i} &= -\log(1-e^{2\pi i z}) 
\end{align*}
we have
\begin{align*}
\im \frac{d}{dz} \frac{\Li (e^{2\pi i z})}{2\pi i} &= - \Arg(1-e^{2\pi i z})
\end{align*}
By the Cauchy Riemann equations, we have
\begin{align*}
 \frac{\partial}{\partial y} \Re \frac{\Li (e^{2\pi i z})}{2\pi i} =  \Arg(1-e^{2\pi i z})
\end{align*}
Therefore, 
\begin{align*}
\Re \frac{\Li(e^{2\pi i z})}{2\pi i} 
&= \frac{\im \Li(e^{2\pi i z})}{2\pi} \\
&= \frac{D(e^{2\pi i z}) - \log |e^{2\pi i z}| \Arg(1-e^{2\pi i z})}{2\pi}\\
&= \frac{D(e^{2\pi i z})}{2\pi} + y\frac{\partial}{\partial y} \Re \frac{\Li(e^{2\pi i z})}{2\pi i} 
\end{align*}
Following the same arguments, it is easy to show that
\begin{align*}
\Re \frac{\Li(e^{-2\pi i z})}{2\pi i} &= \frac{D(e^{-2\pi i z})}{2\pi} + y\frac{\partial}{\partial y} \Re \frac{\Li(e^{-2\pi i z})}{2\pi i} 
\end{align*}
Thus,
\begin{align*}
\Re \frac{\Li(e^{2\pi i(s_2 - 1) - 2\pi i z_1 - 2\pi i z_2})}{2\pi i} &= \frac{D(e^{2\pi i(s_2 - 1) - 2\pi i z_1 - 2\pi i z_2})}{2\pi} + \sum_{k=1}^2 y_k\frac{\partial}{\partial y_k} \Re \frac{\Li(e^{2\pi i(s_2 - 1) - 2\pi i z_1 - 2\pi i z_2)})}{2\pi i} \\
\Re \frac{\Li(e^{2\pi i(s_2 - 1) - 2\pi i z_2})}{2\pi i} &= \frac{D(e^{2\pi i(s_2 - 1) - 2\pi i z_1 - 2\pi i z_2})}{2\pi} +  y_2\frac{\partial}{\partial y_2} \Re \frac{\Li(e^{2\pi i(s_2 - 1)  - 2\pi i z_2)})}{2\pi i} 
\end{align*}
Besides, for the linear terms in $z_1,z_2$, it is straight forward to verify that 
\begin{align*}
\Re\frac{1}{2\pi i} \lt[ \pm \lt(2\pi i \lt( s_1 - 1\rt)\rt) \lt(2\pi i \lt(z_1 -\frac{1}{2} \rt)\rt) \rt]&= y_1 \frac{\partial}{\partial y_1}\Re\frac{1}{2\pi i} \lt[ \pm \lt(2\pi i \lt( s_1 - 1\rt)\rt) \lt(2\pi i \lt(z_1 -\frac{1}{2} \rt)\rt) \rt]\\
\Re\frac{1}{2\pi i}  \lt[ - (2\pi i (s_2-1))(2\pi i z_1 + 2\pi i z_2) \rt]&= \sum_{k=1}^2 y_k \frac{\partial}{\partial y_k} \Re\frac{1}{2\pi i}  \lt[ - (2\pi i (s_2-1))(2\pi i z_1 + 2\pi i z_2) \rt]
\end{align*}
As a result, by direct computation, we get the differential formula.
\end{proof}

Together with the discussion in previous subsection, we have
\begin{align*}
\Re \Phi^{\pm(s_1,s_2)}(\mathbf{z}^{\pm(z_1,z_2)})
= \frac{1}{2\pi} V^{\pm(s_1,s_2)} (\mathbf{z}^{\pm(z_1,z_2)}) 
= \frac{1}{2\pi}\Vol\lt(\SS^3 \backslash WL, u_1 = \pm 2\pi i (1 - s_1), v_2 = 4\pi i(1-s_2)\rt)
\end{align*}
By the Theorem 3 in \cite{NZ85}, since $v_2$ is even in $u_1$ and the volume function is an even function in the variable $u_1$, we have 
$$\Vol\lt(\SS^3 \backslash WL, u_1 = 2\pi i (1 - s_1), v_2 = 4\pi i(1-s_2)\rt)=\Vol\lt(\SS^3 \backslash WL, u_1 = - 2\pi i (1 - s_1), v_2 = 4\pi i(1-s_2)\rt)$$
Thus we have
\begin{align}
\Re \Phi^{\pm(s_1,s_2)}(\mathbf{z}^{\pm(z_1,z_2)})
= \frac{1}{2\pi}\Vol\lt(\SS^3 \backslash WL, u_1 =  2\pi i (1 - s_1), v_2 = 4\pi i(1-s_2)\rt)
\end{align}

Next, from the definitions of the functions $\Phi^{\pm(s_1,s_2)}(\mathbf{z}^{\pm(z_1,z_2)})$, if we regard $s_1,s_2$ as complex variables, then we get two holomorphic functions in $s_1,s_2$. In particular, for each fixed $s_2\in \RR$ sufficiently close to $1$, we have two holomorphic functions for $s_1$ sufficiently close to $1$ given by
$$ N^\pm(s_1) = \Phi^{\pm(s_1,s_2)}(\mathbf{z}^{+(z_1,z_2)}) $$
with the same real part
$$ \Re N^\pm (s_1) = \frac{1}{2\pi}\Vol\lt(\SS^3 \backslash WL, u_1 =  2\pi i (1 - s_1), v_2 = 4\pi i(1-s_2)\rt) $$
Besides, when $s_1=1$, we have $N^+(1)=N^-(1)$. As a result, $N^+(s_1)=N^-(s_1)$ for all real $s_1 \sim 1$. Since this is true for all $s_2\sim 1$, we have
\begin{align}
 \Phi^{+(s_1,s_2)}(\mathbf{z}^{+(z_1,z_2)})
= \Phi^{-(s_1,s_2)}(\mathbf{z}^{-(z_1,z_2)})
\end{align}
for all $s_1,s_2\sim 1$. This completes the proof of Theorem~\ref{mainthmWL}.

Finally, to prove Corollary~\ref{WLVolsym}, note that by Theorem~\ref{mainthmWL}, we have
\begin{align}
\lim_{N\to\infty} \frac{1}{N+\frac{1}{2}}\log\lt|J_{M_1,M_2}(WL, e^{\frac{2\pi i}{N+\frac{1}{2}}})\rt| &= \Vol(\SS^3\backslash WL, u_1 = 2\pi i(1-s_1), v_2= 4\pi i(1-s_2)) \\
\lim_{N\to\infty} \frac{1}{N+\frac{1}{2}}\log\lt|J_{M_2,M_1}(WL, e^{\frac{2\pi i}{N+\frac{1}{2}}})\rt| &= \Vol(\SS^3\backslash WL, u_1 = 2\pi i(1-s_2), v_2= 4\pi i(1-s_1))
\end{align}
By exchanging the components of the Whitehead link, we have 
\begin{align}
\lim_{N\to\infty} \frac{1}{N+\frac{1}{2}}\log\lt|J_{M_1,M_2}(WL, e^{\frac{2\pi i}{N+\frac{1}{2}}})\rt| 
=
\lim_{N\to\infty} \frac{1}{N+\frac{1}{2}}\log\lt|J_{M_2,M_1}(WL, e^{\frac{2\pi i}{N+\frac{1}{2}}})\rt|
\end{align}
This proves (\ref{Volsym1}). Besides, recall that $u_2=0$ if and only if $v_2=0$. (\ref{Volsym2}) follows from (\ref{Volsym1}) by putting $s_2=1$.

\section{Generalization to $J_{M_1,M_2}(W_{a,1,c,d},e^{\frac{2\pi i}{N+\frac{1}{2}}})$}\label{gen}

\subsection{Potential functions for $J_{M_1,M_2}(W_{a,1,c,d},e^{\frac{2\pi i}{N+\frac{1}{2}}})$}
In this section, we compute the potential functions of the $\tilde{J}_{M_1,M_2}(W_{a,1,c,d}, e^{\frac{2\pi i}{N+\frac{1}{2}}})$ with the belt colored by $M_1$ and all the other components colored by $M_2$. First of all, recall from \cite{V08} that the unnormalized $J_{M_1,M_2}(W_{a,1,c,d})$ is given by \footnote{when $c+d=1$, there is an extra factor $t^{\frac{M_2^2-1}{2}}$ coming from the framing. Since it does not affect the exponential growth rate of the colored Jones polynomials, for simplicity we ignore this factor.}
\begin{align*}
J_{M_1,M_2} (W_{a,1,c,d}) = \frac{1}{t^{\frac{1}{2}}-t^{-\frac{1}{2}}} \sum_{n=0}^{M_2-1} (t^{\frac{M_1(2n+1)}{2}}-t^{-\frac{M_1(2n+1)}{2}}) \times t^{an(n+1)} \times \tilde{C}(n,t; M_2)^c \times \tilde{C}(n,t^{-1}; M_2)^d
\end{align*}
where 
$$\tilde{C}(n,t; M_2) =  t^{\frac{M_2(M_2-1)}{2}} \sum_{l = 0}^{M_2-1-n} t^{-M_2(l+n)} \prod_{j=1}^{n} \frac{(1-t^{M_2-l-j})(1-t^{l+j})}{1-t^j}.$$

Note that the contribution of the twists is given by
\begin{align}
t^{an(n+1)} 
&= e^{a \lt(\frac{N+\frac{1}{2}}{2\pi i}\rt) \lt[ 2\pi i \lt(\frac{n}{N+\frac{1}{2}} \rt) \rt] \lt[ 2\pi i \lt( \frac{n}{N+\frac{1}{2}} + \frac{1}{N+\frac{1}{2}} \rt) \rt] } \notag \\
&= e^{a \lt(\frac{N+\frac{1}{2}}{2\pi i}\rt) \lt[ 2\pi i \lt(\frac{n}{N+\frac{1}{2}} - \frac{1}{2} \rt) \rt] \lt[ 2\pi i \lt( \frac{n}{N+\frac{1}{2}} - \frac{1}{2} + \frac{1}{N+\frac{1}{2}} \rt) \rt] }
\times e^{a \pi i - \frac{a}{4}(2\pi i)(N+\frac{1}{2})} \notag \\
&= (-1)^a e^{ - \frac{a}{4}(2\pi i)(N+\frac{1}{2})}\times e^{a \lt(\frac{N+\frac{1}{2}}{2\pi i}\rt) \lt[ 2\pi i \lt(\frac{n}{N+\frac{1}{2}} - \frac{1}{2} \rt) \rt] \lt[ 2\pi i \lt( \frac{n}{N+\frac{1}{2}} - \frac{1}{2} + \frac{1}{N+\frac{1}{2}} \rt) \rt] }\label{t'3}
\end{align}

Besides, for the colored Jones polynomials of the mirror clasp, note that
\begin{align}
 \prod_{j=1}^{n} (1-t^{-(M_2-l-j)}) 
=&\prod_{j=1}^n (1-t^{- M_2 + (N+\frac{1}{2}) +l +j}) \notag\\
=&\prod_{j=1}^n \lt(1 - e^{2\pi i \lt(-\lt(\frac{M_2}{N+\frac{1}{2}} -1\rt) + \frac{l}{N+\frac{1}{2}} + \frac{j}{N+\frac{1}{2}}\rt)}\rt)\notag \\
=&\exp\lt(\frac{N+\frac{1}{2}}{2\pi i}\sum_{j=1}^n \lt( \varphi_r \lt(-\pi \lt(\frac{M_2}{N+\frac{1}{2}} -1\rt) + \frac{l\pi}{N+\frac{1}{2}} + \frac{(j-\frac{1}{2})\pi}{N+\frac{1}{2}}\rt) \rt. \rt.\notag\\
&\lt. \lt.\qquad - \varphi_r\lt(  -\pi\lt(\frac{M_2}{N+\frac{1}{2}} -1\rt) + \frac{l\pi}{N+\frac{1}{2}} + \frac{(j+\frac{1}{2})\pi}{N+\frac{1}{2}}\rt)  \rt) \rt) \notag\\
=& \exp \lt(\frac{N+\frac{1}{2}}{2\pi i} \lt( \varphi_r \lt( -\pi  \lt(\frac{M_2}{N+\frac{1}{2}} -1\rt) + \frac{l\pi}{N+\frac{1}{2}} + \frac{\pi}{2N+1}\rt) \rt.\rt.\notag\\
&\lt. \lt.\qquad - \varphi_r\lt( -\pi  \lt(\frac{M_2}{N+\frac{1}{2}} -1\rt) + \frac{l\pi}{N+\frac{1}{2}} + \frac{n\pi}{N+\frac{1}{2}} + \frac{\pi}{2N+1}\rt)  \rt)\rt)\label{t9}
\end{align}

Besides,
\begin{align}
 \prod_{j=1}^{n} (1-t^{-(l+j)})
=&\prod_{j=1}^n \lt(1 - e^{2\pi i \lt(1 - \frac{l}{N+\frac{1}{2}} - \frac{j}{N+\frac{1}{2}}\rt)}\rt)\notag \\
=&\exp\lt(\frac{N+\frac{1}{2}}{2\pi i}\sum_{j=1}^n \lt( \varphi_r \lt(\pi -\frac{l\pi}{N+\frac{1}{2}} - \frac{(j+\frac{1}{2})\pi}{N+\frac{1}{2}}\rt)  - \varphi_r\lt( \pi -\frac{l\pi}{N+\frac{1}{2}} - \frac{(j-\frac{1}{2})\pi}{N+\frac{1}{2}}\rt)  \rt) \rt)\notag \\
=&\exp\lt(\frac{N+\frac{1}{2}}{2\pi i}\lt( \varphi_r \lt(\pi-\frac{l\pi}{N+\frac{1}{2}} - \frac{n\pi}{N+\frac{1}{2}} - \frac{\pi}{2N+1}\rt)  - \varphi_r\lt( \pi -\frac{l\pi}{N+\frac{1}{2}}-\frac{\pi}{2N+1}\rt)   \rt)\rt)\label{t10}
\\ \notag\\
 \prod_{j=1}^{n} (1-t^{-j})
=&\prod_{j=1}^n \lt(1 - e^{2\pi i \lt(1 - \frac{j}{N+\frac{1}{2}}\rt)}\rt) \notag\\
=&\exp\lt(\frac{N+\frac{1}{2}}{2\pi i} \sum_{j=1}^n \lt( \varphi_r \lt(\pi -\frac{(j+\frac{1}{2})\pi}{N+\frac{1}{2}}\rt)  - \varphi_r\lt(\pi -\frac{(j-\frac{1}{2})\pi}{N+\frac{1}{2}}\rt)  \rt) \rt)\notag \\
=&\exp\lt(\frac{N+\frac{1}{2}}{2\pi i} \lt( \varphi_r\lt( \pi -\frac{n\pi}{N+\frac{1}{2}} - \frac{\pi}{2N+1}\rt) - \varphi_r \lt( \pi - \frac{\pi}{2N+1}\rt) \rt) \rt) \label{t11}
 \end{align}

For each $i = 1,2,\dots,c$, define the function $\psi_{i,M_1,M_2}^{(s_1,s_2)}(z_1,z_{i+1})$ by
\begin{align}
\psi_{i,M_1,M_2}^{(s_1,s_2)} (z_1,z_{i+1})
&= \frac{1}{2\pi i}\lt\{ - \lt(2\pi i \lt(\frac{M_2}{N+\frac{1}{2}} - 1\rt)\rt)(2\pi i z_1 + 2\pi i z_{i+1}) \rt. \notag \\
&\qquad  + \lt[\varphi_r \lt(\frac{M_2\pi}{N+\frac{1}{2}} - \pi z_1 - \pi z_{i+1} -\frac{\pi}{2N+1}\rt) \rt.  \notag\\
&\qquad  - \varphi_r\lt( \frac{M_2\pi}{N+\frac{1}{2}} - \pi z_{i+1}- \frac{\pi}{2N+1}\rt) \notag \\
&\qquad  + \varphi_r \lt(\pi z_{i+1} + \frac{\pi}{2N+1}\rt)  - \varphi_r\lt( \pi z_1 + \pi z_{i+1} + \frac{\pi}{2N+1}\rt) \notag\\
&\lt.\lt.\qquad  + \varphi_r\lt( \pi z_1 + \frac{\pi}{2N+1}\rt) \rt] \rt\}
\end{align}

For each $i = 1,2, \dots, d$, define the function $\kappa_{i,M_1,M_2}^{(s_1,s_2)} (z_1, z_{c+i+1})$ by
\begin{align}
\kappa_{i,M_1,M_2}^{(s_1,s_2)} (z_1,z_{c+i+1})
&= \frac{1}{2\pi i}\lt\{  \lt(2\pi i \lt(\frac{M_2}{N+\frac{1}{2}} - 1\rt)\rt)(2\pi i z_1 + 2\pi i z_{c+i+1}) \rt. \notag \\
&\qquad  -  \lt[\varphi_r \lt( -\pi  \lt(\frac{M_2}{N+\frac{1}{2}} -1\rt) + \pi z_1 + \pi z_{c+i+1} + \frac{\pi}{2N+1}\rt) \rt.  \notag\\
&\qquad  - \varphi_r\lt( -\pi  \lt(\frac{M_2}{N+\frac{1}{2}} -1\rt) + \pi z_{c+i+1} + \frac{\pi}{2N+1}\rt) \notag \\
&\qquad  + \varphi_r \lt( \pi -\pi z_{c+i+1} - \frac{\pi}{2N+1}\rt)  - \varphi_r\lt( \pi - \pi z_1 - \pi z_{c+i+1} - \frac{\pi}{2N+1}\rt) \notag\\
&\lt.\lt.\qquad  + \varphi_r\lt( \pi - \pi z_1 - \frac{\pi}{2N+1}\rt) \rt] \rt\}
\end{align}

Altogether, by (\ref{t1}) - (\ref{t8}) and (\ref{t'3}) - (\ref{t11}), we have
\begin{align}
 J_{M_1, M_2} (W_{a,1,c,d}, e^{\frac{2\pi i}{N+\frac{1}{2}}})
&=   \frac{(-1)^{(M_1 - N+a) } e^{\frac{2\pi i}{N+\frac{1}{2}}(M_2^2 - \frac{M_2}{2} - \frac{1}{2})}  e^{- \frac{a}{4}(2\pi i)(N+\frac{1}{2})} e^{\frac{2\pi i}{N+\frac{1}{2}}\frac{(c-d)M_2(M_2-1)}{2}}}{2\sin(\frac{\pi}{N+\frac{1}{2}})} \notag\\
&\qquad \times 
\exp\lt(  -\varphi_r \lt(  \frac{\pi}{2N+1}\rt)\rt)^{c}
\exp\lt(  \varphi_r \lt(\pi  - \frac{\pi}{2N+1}\rt)\rt)^{d}
\times (I_+ + I_-) 
\end{align}
where
\begin{align}
I_\pm =
&  
\sum_{n=0}^{M_2-1} \sum_{l_1=0}^{M_2-1-n}\dots \sum_{l_c=0}^{M_2-1-n} \sum_{l_1'=0}^{M_2-1-n}\dots \sum_{l_d'=0}^{M_2-1-n} \notag\\
&\qquad \exp\lt( \lt(N+\frac{1}{2}\rt)\Phi^{\pm(s_1, s_2);a,c,d}_{M_1, M_2} 
\lt(\frac{n}{N+\frac{1}{2}},\frac{l_1}{N+\frac{1}{2}},\dots, \frac{l_c}{N+\frac{1}{2}}, \frac{l_1'}{N+\frac{1}{2}},\dots, \frac{l_d'}{N+\frac{1}{2}} \rt) \rt)
\end{align}
with 
\begin{align*}
&\Phi^{\pm (s_1,s_2);a,c,d}_{M_1,M_2} (z_1, z_2, \dots, z_{c+1}, z_{c+2},\dots, z_{c+d+1}) \\
&=  \frac{1}{2\pi i} \lt\{  \pm2\pi i  \lt( \frac{M_1}{N+\frac{1}{2}} - 1\rt) \lt(2\pi i \lt(z_1 -\frac{1}{2} \rt)\rt) 
+ a \lt[2\pi i \lt(z_1-\frac{1}{2} \rt) \rt] \lt[ 2\pi i \lt(z_1-\frac{1}{2}+ \frac{1}{N+\frac{1}{2}}  \rt) \rt]\rt\}  \\
&\qquad  + \sum_{i=1}^c \psi_{M_1,M_2,i}^{(s_1,s_2)} (z_1,z_{i+1}) + \sum_{i=1}^d \kappa_{M_1,M_2,i}^{(s_1,s_2)}(z_1,z_{c+i+1})
\end{align*}

Take $N \to \infty$, we have
\begin{align*}
\psi_{i}^{(s_1,s_2)} (z_1, z_{i+1})
&= \frac{1}{2\pi i} \lt[ - 2\pi i(s_2 -1)(2\pi i z_1 + 2\pi i z_{i+1} )  \rt.\\
&\qquad + \Li\lt(e^{2\pi i (s_2-1) - 2\pi i z_1 - 2\pi i z_{i+1}}\rt) - \Li\lt( e^{2\pi i (s_2-1) - 2\pi i z_{i+1} }\rt) \\
&\qquad \lt. + \Li \lt(e^{2\pi i z_{i+1}}\rt) - \Li\lt(e^{2\pi iz_1+ 2\pi i z_{i+1}}\rt) + \Li\lt(e^{2\pi i z_1}\rt) \rt] ,\\
\kappa_i^{(s_1,s_2)}(z_1,z_{c+i+1}) 
&= \frac{1}{2\pi i} [ 2\pi i(s_2 -1)(2\pi i z_1 + 2\pi i z_{c+i+1} ) \\
&\qquad - \Li\lt(e^{-2\pi i (s_2-1) + 2\pi i z_1 + 2\pi i z_{c+i+1}}\rt) + \Li\lt( e^{-2\pi i (s_2-1) + 2\pi i z_{c+i+1} }\rt) \\
&\qquad - \Li \lt(e^{-2\pi i z_{c+i+1}}\rt) + \Li\lt(e^{-2\pi iz_1- 2\pi i z_{c+i+1}}\rt) - \Li\lt(e^{-2\pi i z_1}\rt) ]
\end{align*}
and
\begin{align*}
&\Phi^{\pm (s_1,s_2);a,c,d}  (z_1, z_2,z_3, \dots, z_{c+1}, z_{c+2},\dots, z_{c+d+1})  \\
&=  \frac{1}{2\pi i} \lt\{  \pm2\pi i  \lt( s_2 - 1\rt) \lt(2\pi i \lt(z_1 -\frac{1}{2} \rt)\rt)
 + a \lt[2\pi i \lt(z_1-\frac{1}{2} \rt) \rt]^2  \rt\} \\
&\qquad + \sum_{i=1}^c \psi^{(s_1,s_2)}_i (z,z_{i+1}) + \sum_{i=1}^d \kappa_i^{(s_1,s_2)}(z_1,z_{c+i+1})
\end{align*}

As an analogue of Lemma~\ref{EM_2}, we let
\begin{align}
R_{M_2}(z_1,z_2) 
&= \frac{1}{2}\lt[ \log\lt(1-e^{-2\pi i \lt(\frac{M_2}{N+\frac{1}{2}}-1\rt) + 2\pi i z_1 + 2\pi i z_2 }\rt)
- \log\lt(1-e^{-2\pi i \lt(\frac{M_2}{N+\frac{1}{2}}-1\rt) + 2\pi i z_2 } \rt) \rt.\notag\\
&\lt.\qquad - \log\lt(1- e^{-2\pi i z_2 } \rt)
+ \log\lt(1- e^{-2\pi iz_1- 2\pi i z_2} \rt) 
- \log\lt(1-  e^{-2\pi i z_1 } \rt) \rt]
\end{align}

Following the proof of Lemma~\ref{EM_2}, we have

\begin{lemma}\label{EM_2'}
On any compact subset of 
$$\lt\{ (z_1,z_2,\dots,z_{c+d+1}) \in \CC \mid \Re z_1 , \dots, \Re z_{c+d+1} >0, \Re z_1 + \Re z_k < s_2 \text{ for }k=2,\dots, c+d+1\rt\},$$
we have
\begin{align}
&\Phi^{\pm(s_1,s_2);a,c,d}_{M_1,M_2} (z_1, z_2,z_3, \dots, z_{c+1}, z_{c+2},\dots, z_{c+d+1}) \notag\\
&= \Phi^{\pm (s_1,s_2);a,c,d}  (z_1, z_2,z_3, \dots, z_{c+1}, z_{c+2},\dots, z_{c+d+1}) \notag\\
&\qquad+\frac{1}{N+\frac{1}{2}}E_{M_2}^{a,c,d} (z_1, z_2,z_3, \dots, z_{c+1}, z_{c+2},\dots, z_{c+d+1})+ O\lt(\frac{1}{\lt(N+\frac{1}{2}\rt)^2}\rt),
\end{align}
where
\begin{align}
&\quad E_{M_2}^{a,c,d} (z_1, z_2,z_3, \dots, z_{c+1}, z_{c+2},\dots, z_{c+d+1})\notag\\
&= -4\pi^2 a \lt(z_1-\frac{1}{2}\rt) + \sum_{i=1}^c E_{M_2}(z_1,z_i+1) + \sum_{i=1}^d R_{M_2}(z_1,z_{c+i+1})
\end{align}
\end{lemma}

Note that 
\begin{align}
\frac{\partial}{\partial z_1}  \psi_i (z_1,z_{i+1})
&=  \log(1-e^{2\pi i (s_2-1)}e^{-2\pi i z_1 - 2\pi i z_{i+1}}) + \log(1 - e^{2\pi i z_1 + 2\pi i z_{i+1} }) - \log(1 - e^{2\pi i z_1}) \notag \\
& \qquad  -2\pi i(s_2 - 1)\\
\frac{\partial}{\partial z_{i+1}}  \psi_i (z_1,z_{i+1})
&= \log(1-e^{2\pi i (s_2 -1)}e^{-2\pi i z_1 - 2\pi i z_{i+1}}) + \log(1 - e^{2\pi i z_1 + 2\pi i z_{i+1}}) \notag\\
&\quad   - \log(1 - e^{2\pi i (s_2 - 1) - 2\pi i z_{i+1}}) - \log(1 - e^{2\pi i z_{i+1}})  -2\pi i(s_2 - 1)  \\
\frac{\partial}{\partial z_1}  \kappa_i (z_1,z_{c+i+1})
&=  \log(1-e^{-2\pi i (s_2-1)}e^{2\pi i z_1 + 2\pi i z_{c+i+1}}) + \log(1 - e^{-2\pi i z_1 - 2\pi i z_{c+i+1} }) \notag\\
& \qquad - \log(1 - e^{-2\pi i z_1})  +2\pi i(s_2 - 1)\\
\frac{\partial}{\partial z_{c+i+1}}  \kappa_i (z_1,z_{c+i+1})
&= \log(1-e^{-2\pi i (s_2 -1)}e^{2\pi i z_1 + 2\pi i z_{c+i+1}}) + \log(1 - e^{-2\pi i z_1 - 2\pi i z_{c+i+1}}) \notag\\
&\quad   - \log(1 - e^{-2\pi i (s_2 - 1) + 2\pi i z_{c+i+1}}) - \log(1 - e^{-2\pi i z_{c+i+1}})  + 2\pi i(s_2 - 1)  
\end{align}

As a result, the critical point equations for the potential function $\Phi^{\pm (s_1,s_2);a,c,d}(z_1, z_2,z_3, \dots, z_{c+d+1})$ are given by

\begin{empheq}[left = \empheqlbrace]{align}
&\quad \sum_{i=1}^c \lt[\log(1-e^{2\pi i (s_2 -1)}e^{-2\pi i z_1 - 2\pi i z_{i+1}}) + \log(1 - e^{2\pi i z_1 + 2\pi i z_{i+1}}) - \log(1 - e^{2\pi i z_1}) \rt] \notag \\
&\quad + \sum_{i=1}^d \lt[\log(1-e^{-2\pi i (s_2 -1)}e^{2\pi i z_1 + 2\pi i z_{c+i+1}}) + \log(1 - e^{-2\pi i z_1 - 2\pi i z_{c+i+1}}) - \log(1 - e^{-2\pi i z_1}) \rt] \notag \\
&=  2\pi i\lt[ \mp\lt( s_1 - 1\rt) + (c-d)(s_2 - 1) - 2a \lt(z_1-\frac{1}{2}\rt)\rt] \label{sbm1}\\ \notag \\
&\quad \log(1-e^{2\pi i (s_2 -1)}e^{-2\pi i z_1 - 2\pi i z_{i+1}}) + \log(1 - e^{2\pi i z_1 + 2\pi i z_{i+1}}) \notag\\
&\quad  - \log(1 - e^{2\pi i (s_2 - 1) - 2\pi i z_{i+1}}) - \log(1 - e^{2\pi i z_{i+1}})   \notag\\
&=  2\pi i  (s_2 - 1) \quad\text{for $i=1,2,\dots, c$} \\ \notag \\
&\quad \log(1-e^{-2\pi i (s_2 -1)}e^{2\pi i z_1 + 2\pi i z_{c+i+1}}) + \log(1 - e^{-2\pi i z_1 - 2\pi i z_{c+i+1}}) \notag\\
&\quad  - \log(1 - e^{-2\pi i (s_2 - 1) + 2\pi i z_{c+i+1}})- \log(1 - e^{-2\pi i z_{c+i+1}})  \notag \\
&=  -2\pi i  (s_2 - 1) \quad\text{for $i=1,2,\dots, d$} 
\end{empheq}

Put $Z_{i}=e^{2\pi i z_{i}}$, $U_i = \frac{1}{Z_1 Z_{i+1}}$ and $B_l = e^{2\pi i s_l}$ for $i=1,2,\dots, c+d$ and $l = 1,2$. After taking exponential, we have

\begin{empheq}[left = \empheqlbrace]{align}\hspace{-2pt}
\lt[\prod_{i=1}^c  \frac{\lt(1 - B_2  U_i \rt)\lt( 1 - \dfrac{1}{U_i}\rt)}{B_2\lt( 1- Z_1 \rt)} \rt]
\lt[\prod_{i=1}^d  \frac{\lt(1 - B_2^{-1}  U_{i+c}^{-1} \rt)\lt( 1 - \dfrac{1}{U_{i+c}^{-1}}\rt)}{B_2^{-1}\lt( 1- Z_1^{-1} \rt)} \rt]
 (Z_1^2)^a&= B_1^{\mp} \\
 \frac{(1- B_2 U_i)\lt(1-\dfrac{1}{U_i}\rt)}{(1-\dfrac{B_2 }{Z_{i+1}})(1- Z_{i+1})} &= B_2  \quad\text{for $i=1,2,\dots,c$}\\
 \frac{(1- B_2^{-1} U_{c+i}^{-1})\lt(1-\dfrac{1}{U_{c+i}^{-1}}\rt)}{(1-\dfrac{B_2^{-1} }{Z_{c+i}^{-1}})(1- Z_{c+i}^{-1})} &= B_2^{-1}  \quad\text{for $i=1,2,\dots,d$}
\end{empheq}

\subsection{Asymptotics of $J_{N, N}(W_{a,1,c,d},e^{\frac{2\pi i}{N+\frac{1}{2}}})$}
When $M_1=M_2=N$, the potential function becomes
\begin{align*}
&\Phi^{(1,1);a,c,d}  (z_1, z_2,z_3, \dots, z_{c+d+1})\\ 
&=  \frac{1}{2\pi i} \lt(  a \lt(2\pi i \lt(z_1-\frac{1}{2} \rt) \rt)^2  \rt)
 + \sum_{i=1}^c \psi^{(1,1)}_i (z_1,z_{i+1}) + \sum_{i=1}^d \kappa_i^{(1,1)}(z_1,z_{c+i+1})
\end{align*}
Besides, the critical point equations of the potential function are given by
\begin{empheq}[left = \empheqlbrace]{align}
& \sum_{i=1}^c \lt[\log(1-e^{-2\pi i z_1 - 2\pi i z_{i+1}}) + \log(1 - e^{2\pi i z_1 + 2\pi i z_{i+1}}) - \log(1 - e^{2\pi i z_1}) \rt] \notag\\
& - \sum_{i=1}^{d} \lt[\log(1-e^{2\pi i z_1 + 2\pi i z_{c+i+1}}) + \log(1 - e^{-2\pi i z_1 - 2\pi i z_{c+i+1}}) - \log(1 - e^{-2\pi i z_1}) \rt] \notag\\
&=   -4\pi i a(z_1 - \frac{1}{2} )  \label{a1cd1} \\\notag\\
& \log(1-e^{-2\pi i z_1 - 2\pi i z_{i+1}}) + \log(1 - e^{2\pi i z_1 + 2\pi i z_{i+1}}) - \log(1 - e^{ - 2\pi i z_{i+1}})
 - \log(1 - e^{2\pi i z_{i+1}}) \notag\\
& = 0 \quad\text{for $i=1,2,\dots, c$} \label{a1cd2} \\\notag\\
& \log(1-e^{2\pi i z_1 + 2\pi i z_{c+i+1}}) + \log(1 - e^{-2\pi i z_1 - 2\pi i z_{c+i+1}}) - \log(1 - e^{  2\pi i z_{i+1}})
 - \log(1 - e^{-2\pi i z_{c+i+1}}) \notag\\
& = 0 \quad\text{for $i=1,2,\dots, d$} \label{a1cd3} 
\end{empheq}
It is straightforward to verify that $(z_1,z_2,\dots,z_c, z_{c+1},\dots, z_{c+d+1}) = \lt(\frac{1}{2},\frac{1}{4},\dots, \frac{1}{4}\rt)$ is a solution of equations (\ref{a1cd1}) - (\ref{a1cd3}). Besides, it is the unique maximum point on the region
$$ D = \{ (z_1,\dots, z_{c+d+1})\in [0,1]^{c+d+1} \mid z_1 + z_i \leq 1 \text{ for $i=2,\dots, c+d+1$} \} $$
Moreover, since
\begin{align}
\psi^{(1,1)}_i \lt( \frac{1}{2}, \frac{1}{4} \rt) &= \frac{1}{2\pi i} [2\Li (i) - 2\Li(-i) + \Li(-1)] = \frac{1}{2\pi} \lt[ 8L \lt(\frac{\pi}{4}\rt) + \frac{\pi^2 i}{12} \rt] \\
\kappa^{(1,1)}_i \lt( \frac{1}{2}, \frac{1}{4} \rt) &= \frac{1}{2\pi i} [2\Li (i) - 2\Li(-i) - \Li(-1)] = \frac{1}{2\pi} \lt[ 8L \lt(\frac{\pi}{4}\rt) - \frac{\pi^2 i}{12} \rt] 
\end{align}
the critical value is given by
\begin{align*}
\Phi^{(1,1);a,c,d} \lt(\frac{1}{2},\frac{1}{4},\dots, \frac{1}{4}\rt)
= \frac{1}{2\pi i} \sum_{i=1}^c \psi_i \lt(\frac{1}{2},\frac{1}{4}\rt) + \frac{1}{2\pi i} \sum_{i=1}^d \kappa_i \lt(\frac{1}{2},\frac{1}{4} \rt)
= \frac{1}{2\pi} \lt[8(c+d)L\lt(\frac{\pi}{4}\rt) + (c-d)\frac{\pi^2 i}{12} \rt]
\end{align*}
Besides, by Lemma A.3 of \cite{O52}, we have
\begin{align}\label{QDFcoe}
\exp\lt(-\phi^h \lt( \frac{\pi}{2N+1} \rt) \rt) 
&\stackrel[N \to \infty]{\sim}{ } e^{-\frac{\pi i}{4}}(N+\frac{1}{2})^{-\frac{1}{2}}\exp\lt( \frac{N+\frac{1}{2}}{2\pi} \frac{\pi^2 i}{6} \rt) \\
\exp\lt(\phi^h \lt(\pi - \frac{\pi}{2N+1} \rt) \rt) 
&\stackrel[N \to \infty]{\sim}{ } e^{\frac{\pi i}{4}}(N+\frac{1}{2})^{-\frac{1}{2}}\exp\lt( \frac{N+\frac{1}{2}}{2\pi} \lt(\frac{-\pi^2 i}{6}\rt) \rt)
\end{align}

Nevertheless, the Hessian of $\Phi^{(1,1);a,c,d}(\frac{1}{2},\frac{1}{4},\dots, \frac{1}{4})$ is given by
$$   
\Hess\lt(\Phi^{(1,1);a,c,d}\lt(\frac{1}{2},\frac{1}{4},\dots, \frac{1}{4}\rt)\rt) = 2\pi i
\begin{pmatrix}
    (c+d)i - \frac{c-d}{2} +2a & i  & i & i &\dots  & i\\
    i & 2i  & 0 & 0 &\dots  &0\\
    i & 0& 2i &0 &\dots  &0\\
    &&&\vdots\\
    i & 0 & 0 &\dots &0 & 2i\\
\end{pmatrix}
$$
For any $a_1,a_2,a_3 \in \CC$, one can verify that the determinant of the $(c+d+1)\times (c+d+1)$ matrix
$$
\begin{pmatrix}
a_1&a_2&a_2&a_2&\dots&a_2 \\
a_2&a_3&0&0&\dots&0 \\
a_2&0&a_3&0&\dots&0 \\
\vdots&\vdots&\vdots&\vdots&\vdots\\
a_2&0&0&0&\dots&0 \\
a_2&0&0&0&\dots&a_3
\end{pmatrix}
$$
is given by
$$
\begin{vmatrix}
a_1&a_2&a_2&a_2&\dots&a_2 \\
a_2&a_3&0&0&\dots&0 \\
a_2&0&a_3&0&\dots&0 \\
\vdots&\vdots&\vdots&\vdots&\vdots\\
a_2&0&0&0&\dots&0 \\
a_2&0&0&0&\dots&a_3
\end{vmatrix}
=
a_3^{c+d-1}[a_1a_3-(c+d)a_2^2]
$$
In particular, if we put $a_1=(c+d)i - \frac{c-d}{2}+2a, a_2=i$ and $a_3=2i$, one can show that
\begin{align*}
\det\lt(-\Hess\lt(\Phi^{(1,1);a,c,d}\lt(\frac{1}{2},\frac{1}{4},\dots, \frac{1}{4}\rt)\rt) \rt)
\neq 0
\end{align*}
for any $a\in \ZZ, c,d\in \NN\cup\{0\}$ with $c+d\geq 1$. Theorem~\ref{CJWNN} then follows by similar argument as in Section~\ref{WLNN}.

\subsection{Triangulation of the $\SS^3\backslash W_{0,1,c,0}$}
The triangulation of $\SS^3 \backslash WL_{0,1,c,0}$ is similar to that for the Whitehead link complement. First of all, we prepare $c$ ideal octahedra. Next, we glue the top of the cylinder to the bottom of the next cylinder by the rule $C_2 \to C_1$ and $A_2 \to A_1$ (Figure \ref{f14}), then we obtain a decomposition of the boundary torus into paralleograms. Note that this torus is exactly the boundary torus of a tubular neighborhood of the belt component of $WL_{0,1,c,0}$.\\ 

          \begin{figure}[H]
          \centering
              \includegraphics[width=0.8\linewidth]{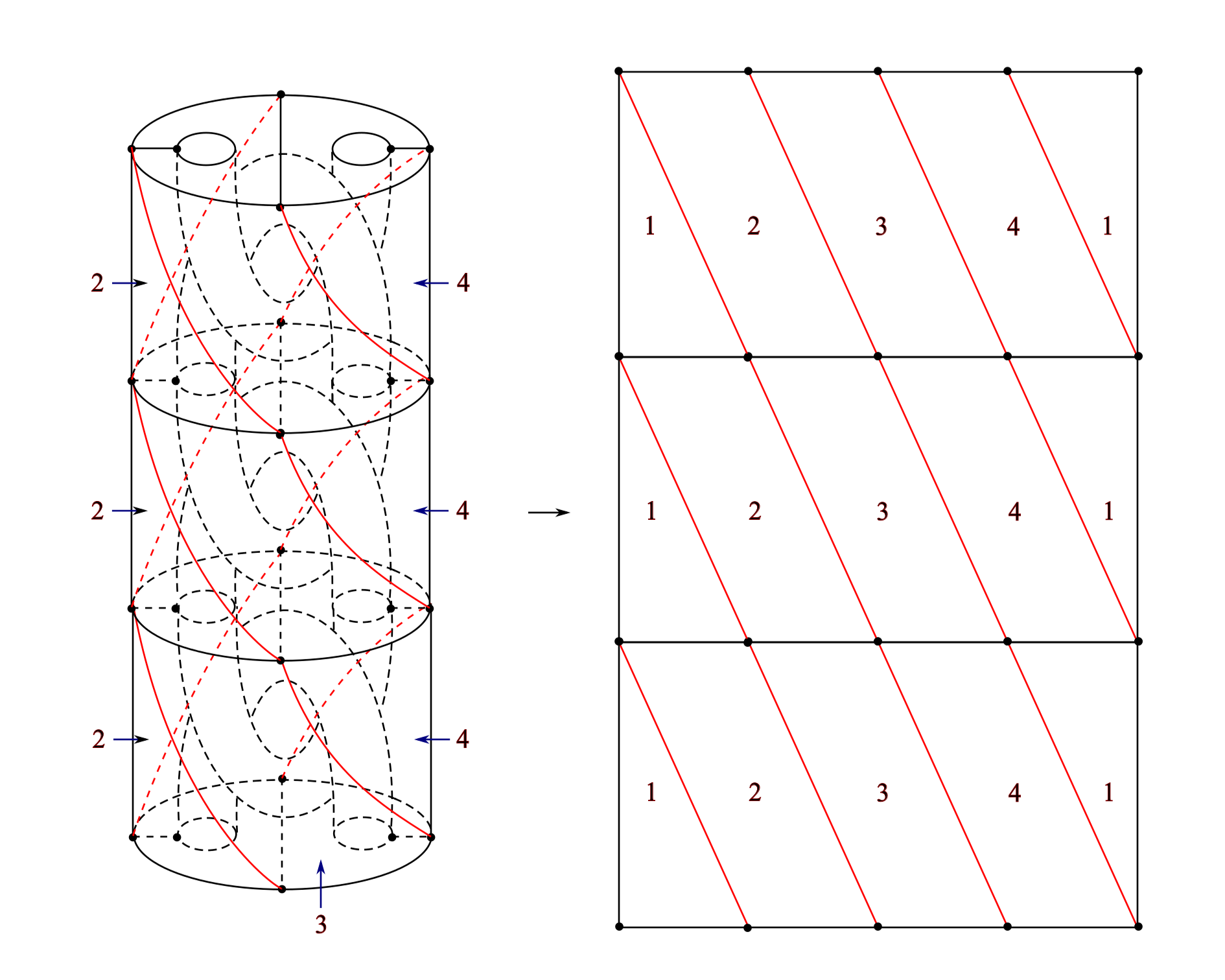}
              \caption{decomposition of the boundary torus into parallelograms}\label{f14}
          \end{figure}

Originally, each tetrahedron has its own assignment of shape parameters $Z_1=e^{2\pi i z_1}, W_i=e^{2\pi i z_{i+1}}$ and $U_i=\frac{1}{Z_1W_i}$. If we put $z_1=z_2=\dots=z_{c+1}=z$, $w_1=w_1=\dots=w_{c+1}=w$ and $u_1=u_2=\dots=u_{c+1}=u$, by similar calculation in Section~\ref{geoWL}, one can verify that the edges equations can be reduced to a single equation
\begin{align}
\lt(\frac{U'}{W'} \rt) \lt( \frac{(B_2 U)''}{(B_2^{-1}W)''}\rt) = 1
\end{align}

%%%%%%%%%%%%%%%%%%%
\iffalse
Moreover, recall from Section~\ref{geoWL} that there exist a family of solution $(z,w)=\lt(z(s_1,s_2), w(s_1,s_2)\rt)$ near $(s_1,s_2)=(1,1)$ satisfying $(z(1,1),w(1,1))=(\frac{1}{2},\frac{1}{4})$ and
\begin{empheq}[left = \empheqlbrace]{align}
&\quad \log(1-e^{2\pi i (s_2 -1)}e^{-2\pi i z_1 - 2\pi i z_2}) + \log(1 - e^{2\pi i z_1 + 2\pi i z_2}) - \log(1 - e^{2\pi i z_1}) \notag \\
&=  2\pi i\lt[-(s_1 - 1) + (s_2 - 1) \rt] \label{solve1}\\ \notag \\
&\quad \log(1-e^{2\pi i (s_2 -1)}e^{-2\pi i z_1 - 2\pi i z_2}) + \log(1 - e^{2\pi i z_1 + 2\pi i z_2}) - \log(1 - e^{2\pi i (s_2 - 1) - 2\pi i z_2})\notag\\
&\quad  - \log(1 - e^{2\pi i z_2}) \label{solve2}  \notag \\
&=  2\pi i  (s_2 - 1)
\end{empheq}
\fi
%%%%%%%%%%%%%%%%%%%

On the other hand, the critical point equations for the potential function $\Phi^{(1,1);0,c,0}  (z_1, z_2,z_3, \dots, z_{c+1})$ are given by

\begin{empheq}[left = \empheqlbrace]{align}
&\quad \sum_{i=1}^c \lt[\log(1-e^{2\pi i (s_2 -1)}e^{-2\pi i z_1 - 2\pi i z_{i+1}}) + \log(1 - e^{2\pi i z_1 + 2\pi i z_{i+1}}) - \log(1 - e^{2\pi i z_1}) \rt] \notag \\
&=  2\pi i\lt[ -\lt( s_1 - 1\rt) + c(s_2 - 1) \rt] \label{W01c01}\\ \notag \\
&\quad \log(1-e^{2\pi i (s_2 -1)}e^{-2\pi i z_1 - 2\pi i z_{i+1}}) + \log(1 - e^{2\pi i z_1 + 2\pi i z_{i+1}}) \notag\\
&\quad  - \log(1 - e^{2\pi i (s_2 - 1) - 2\pi i z_{i+1}}) - \log(1 - e^{2\pi i z_{i+1}})   \notag\\
&=  2\pi i  (s_2 - 1) \label{W01c02}\quad\text{for $i=1,2,\dots, c$} 
\end{empheq}
By similar argument in Section~\ref{geoWL}, one can show that the holonomy of the meridian of the belt component and the holonomy of the longitude of every clasps components are exactly $B_1$ and $B_2$ respectively. Theorem~\ref{mainthmWa1c0} then follows from similar arguments in Section~\ref{WLs}-\ref{geoWL}.

\section{Volume conjecture for $TV_r(W_{a,b,c,d}, e^{\frac{2\pi i}{r}})$ with $b\geq 1$}\label{TV}
Let $r=2N+1$. Recall from Theorem~\ref{relationship} that the $r$-th Turaev-Viro invariants for the link complement $\SS^3 \backslash L$ is related to the colored Jones polynomials of the link as follows:
\begin{align*}
TV_{r}\left(\SS^3 \backslash L, e^{\frac{2\pi i}{r}}\right) = 2^{n-1}\lt( \frac{2\sin(\frac{2\pi}{r})}{\sqrt{r}}\rt)^{2} \sum_{1\leq \vec{M} \leq \frac{r-1}{2}}\left|{J}_{\vec{M}}\left(L,e^{\frac{2\pi i }{N+\frac{1}{2}}}\right)\right|^2 ,
\end{align*}
From the formula of $J_{\vec{M}}(W_{a,b,c,d},e^{\frac{2\pi i}{N+\frac{1}{2}}})$, it is easy to see that the colored Jones polynomials of the belt tangle and twist tangle grow at most polynomially. Thus, in order to find an upper bound for the exponential growth rate of $J_{\vec{M}}(W_{a,b,c,d},e^{\frac{2\pi i}{N+\frac{1}{2}}})$, we only need to find that for the clasp tangle $\tilde{C}(n,e^{\frac{2\pi i}{N+\frac{1}{2}}}; M_2)$. 
 
\begin{lemma}\label{uppbdW}
For any $n \in \{1,2,\dots,M-1\}$, $l \in \{1,2,\dots, M-1-n\}$, let 
$$ c_{M}(n,l;t) = \prod_{j=1}^{n} \lt|\frac{(1-t^{M-l-j})(1-t^{l+j})}{1-t^j}\rt| $$
For each $M$, let $n_M\in \{1,2,\dots,M-1\}$ and $l_M \in \{1,2,\dots, M-1-n\}$ such that $c_M(n_M,l_M)$ achieves the maximum among all $c_{M}(n,l)$. Assume that $\frac{M}{N+\frac{1}{2}} \to s \in [0,1]$, $\frac{n_M}{N+\frac{1}{2}} \to n_s$ and $\frac{l_M}{N+\frac{1}{2}} \to l_s$. Then we have
$$ \lim_{N \to \infty} \frac{1}{N+\frac{1}{2}} \log(c_M(n,l;e^{\frac{2\pi i}{N+\frac{1}{2}}})) \leq \frac{\Vol(\SS^3\backslash WL)}{2\pi} $$
Furthermore, the equality holds if and only if 
$ s=1, n_s = \frac{1}{2} \text{ and } l_s = \frac{1}{4}$.
\end{lemma}

\begin{proof}
Define 
$$ f(x,y,s) = \frac{1}{\pi}[\Lambda(\pi s - \pi x - \pi y) - \Lambda(\pi s - \pi y) + \Lambda(\pi y) - \Lambda(\pi x + \pi y) + \Lambda(\pi x)],$$
where $\Delta_s=\{(x,y,s) \in \RR^3 \mid 0\leq x,y \leq \pi, 0\leq x + y \leq s, s\in [0,1]\}$. One can show that when
$$s=\lim_{N\to \infty} \frac{M}{N+\frac{1}{2}}, x=\lim_{N\to\infty} \frac{n}{N+\frac{1}{2}} \text{ and } y=\lim_{N\to\infty} \frac{l}{N+\frac{1}{2}},$$
we have
$$ \lim_{N \to \infty} \frac{1}{N+\frac{1}{2}} \log(c_M(n,l;e^{\frac{2\pi i}{N+\frac{1}{2}}})) = f(x,y,s) $$
In order to find out the maximum of $f$ inside the region $\Delta_s$, we will first find out all the critical point of $f$ on $\Delta_s$, and then estimate the value of $f$ along the boundary. Note that
\begin{empheq}[left = \empheqlbrace]{align}
f_x &= \log|2\sin(\pi s - \pi x - \pi y)| + \log|2\sin(\pi x + \pi y)| - \log|2\sin (\pi x)| \\
f_y &= \log|2\sin(\pi s - \pi x - \pi y)| - \log|2 \sin (\pi s - \pi y)| - \log|2\sin (\pi y)| + \log|2\sin(\pi x + \pi y)| \\
f_s &= -\log|2\sin(\pi s - \pi x - \pi y)| + \log|2\sin(\pi s - \pi y)|
\end{empheq}

\begin{enumerate}
\item To find out the critical point of $f$, note that
\begin{alignat}{2}
f_s = 0 
&\implies \log|2\sin(\pi s - \pi x - \pi y)| = \log|2\sin(\pi s - \pi y)|  \label{*}\\
&\implies \qquad\qquad\qquad\qquad x=0 \quad\text{or} \quad x+2y = 2s-1 \label{**}
\end{alignat}
By putting (\ref{*}) into the equation $f_y=0$, we get $x=0$ or $x+2y=1$. As a result, $f_y=f_s=0$ imply $x=0$ or $s=1$. In both cases, the critical point lies on the boundary.

\item To estimate the value of $f$ along the boundary, note that
\begin{enumerate}
\item on the boundary where $s=1$, $(x,y,s) = \lt(\frac{1}{2}, \frac{1}{4}, 1 \rt)$ is a unique maximum point with maximum value 
$f\lt(\frac{1}{2}, \frac{1}{4}, 1 \rt) = \frac{1}{2\pi} \Vol(\SS^3 \backslash WL)$.

\item on the boundary where $x=0$, we have $f(0,y,s)=0$.

\item when $y=0$ or $x+y=s$, the function $f$ is given by
\begin{empheq}[left = \empheqlbrace]{align*}
f(x,0,s) &= \frac{1}{\pi} [\Lambda(\pi s - \pi x) - \Lambda(\pi s)] \\
f(x,s-x,s) &= \frac{1}{\pi} [\Lambda((\pi y) - \Lambda(\pi s)]
\end{empheq}
In both cases, we have
$$|f| \leq \frac{1}{\pi} \lt[ 2\Lambda\lt(\frac{\pi}{3}\rt)\rt] \leq \frac{1}{2\pi}\lt[6\Lambda\lt(\frac{\pi}{3}\rt)\rt] = \frac{1}{2\pi} \Vol(\SS^3 \backslash 4_1) < \frac{1}{2\pi} \Vol(\SS^3 \backslash WL)$$
\end{enumerate}
where the last inequality follows from the fact that the figure eight knot complement can be obtained by doing surgery along the belt component of the Whitehead link. This completes the proof.
\end{enumerate}
\end{proof}

From Theorem~\ref{relationship} and Lemma~\ref{uppbdW}, we can see that
\begin{align}
\limsup_{\substack{ r\to \infty \\r \text{ odd}}}\frac{2\pi}{r}\log \left(TV_{r}(\SS^3\backslash W_{a,b,c,d},e^{\frac{2\pi i}{r}})\right) 
\leq (c+d)\Vol(\SS^3\backslash WL) = \Vol (\SS^3\backslash W_{a,b,c,d})
\end{align}

Besides, from Theorem~\ref{relationship} and Theorem~\ref{CJWNN}, for $r=2N+1$, since 
\begin{align}
\frac{2\pi}{r}\log \left(TV_{r}(\SS^3\backslash W_{a,b,c,d},e^{\frac{2\pi i}{N+\frac{1}{2}}})\right) 
&\geq \frac{2\pi}{r}\log \lt| J_{\vec{N}}\left((\SS^3\backslash W_{a,1,c,d},e^{\frac{2\pi i}{r}})\right) \rt|^2 
\end{align}
and
\begin{align}
\lim_{N\to \infty}\frac{2\pi}{2N+1}\log \lt| J_{\vec{N}}\left((\SS^3\backslash W_{a,1,c,d},e^{\frac{2\pi i}{N+\frac{1}{2}}})\right) \rt|^2 &= \Vol (\SS^3\backslash W_{a,1,c,d}) = v_3||\SS^3\backslash W_{a,b,c,d}||,
\end{align}
Corollary \ref{TVW} follows from squeeze theorem.

\end{document}